\documentclass[11pt]{amsart}

\usepackage{amsmath,amsfonts,amssymb,amsthm,mathtools}

\usepackage{verbatim}
\usepackage{amsmath,amsfonts}
\usepackage[mathscr]{euscript} 
\usepackage{amsthm}
\usepackage{blkarray,bigstrut} 
\usepackage{url}

\usepackage{graphicx} 

\usepackage{enumitem} 

\usepackage{hyperref} 
\hypersetup{colorlinks} 

\usepackage[colorinlistoftodos]{todonotes}

\RequirePackage{cleveref}
\usepackage{hypcap}
\hypersetup{colorlinks=true, citecolor=darkblue, linkcolor=darkblue}
\definecolor{darkblue}{rgb}{0.0,0,0.7}
\newcommand{\darkblue}{\color{darkblue}}

\definecolor{darkred}{rgb}{0.68,0,0}
\newcommand{\darkred}{\color{darkred}}

\definecolor{darkgreen}{rgb}{0,.38,0}
\newcommand{\darkgreen}{\color{darkgreen}}

\newcommand{\defn}[1]{\emph{\darkblue #1}}
\newcommand{\defna}[1]{\emph{\darkred #1}}
\newcommand{\defnb}[1]{\emph{\darkblue #1}}
\newcommand{\defng}[1]{\emph{\darkgreen #1}}

\setlist[enumerate]{
	label=\textnormal{({\roman*})},
	ref={\roman*}}

\makeatletter
\def\th@plain{%
	\thm@notefont{}
	\itshape 
}
\def\th@definition{%
	\thm@notefont{}
	\normalfont 
}
\makeatother

\makeatletter
\def\fdsy@scale{1}
\newcommand\fdsy@mweight@normal{Book}
\newcommand\fdsy@mweight@small{Book}
\newcommand\fdsy@bweight@normal{Medium}
\newcommand\fdsy@bweight@small{Medium}

\DeclareFontFamily{U}{FdSymbolB}{}
\DeclareFontShape{U}{FdSymbolB}{m}{n}{
	<-7.1> s * [\fdsy@scale] FdSymbolB-\fdsy@mweight@small
	<7.1-> s * [\fdsy@scale] FdSymbolB-\fdsy@mweight@normal
}{}
\DeclareFontShape{U}{FdSymbolB}{b}{n}{
	<-7.1> s * [\fdsy@scale] FdSymbolB-\fdsy@bweight@small
	<7.1-> s * [\fdsy@scale] FdSymbolB-\fdsy@bweight@normal
}{}

\makeatother

\DeclareSymbolFont{fdrelations}{U}{FdSymbolB}{m}{n}
\SetSymbolFont{fdrelations}{bold}{U}{FdSymbolB}{b}{n}

\DeclareMathSymbol{\lescc}{\mathrel}{fdrelations}{66}

\newtheorem{thm}{Theorem}[section]
\newtheorem{lemma}[thm]{Lemma}

\newtheorem{conv}[thm]{Convention}

\theoremstyle{definition}

\newtheorem{rem}[thm]{Remark}
\newtheorem{definition}[thm]{Definition}

\numberwithin{figure}{section}
\numberwithin{equation}{section}

\def\wh{\widehat}

\def\zz{\mathbb Z}
\def\AA{\mathbb A}
\def\nn{\mathbb N}

\def\rr{\mathbb R}

\def\rrs{\mathbb R_{>0}}
\def\rrp{\mathbb R_{\ge 0}}

\def\sm{\smallsetminus}

\def\la{\lambda}

\def\ep{\varepsilon}
\def\al{\alpha}
\def\be{\beta}

\def\ve{\varepsilon}

\def\cI{\mathcal I}

\def\cL{\mathcal L}

\def\cP{\mathcal P}

\def\ssu{\subset}

\def\<{\langle}
\def\>{\rangle}

\def\rI{ {\text {\sc I} } }

\def\rk{\textnormal{rk}}

\def\0{{\mathbf 0}}

\def\.{\hskip.06cm}
\def\ts{\hskip.03cm}

\def\sts{\hskip.015cm}

\def\area{{\text {\rm area} }}

\def\ba{\textbf{\textit{a}}}

\def\nin{\noindent}

\def\aD{\textrm{D}}

\def\ah{\textrm{h}}
\def\aI{{ \textrm{I} } }

\def\aM{\textrm{M}}
\def\am{\textrm{m}}

\def\an{\textrm{n}}

\newcommand\ar{\textnormal{d}} 

\def\av{\textrm{v}}

\def\bA{\textbf{\textrm{A}}\hskip-0.03cm{}}
\def\bAr{\textbf{\textrm{\em A}}\hskip-0.03cm{}}

\def\bB{\textbf{\textrm{B}}\hskip-0.03cm{}}

\def\bD{\textbf{\textrm{D}}\hskip-0.03cm{}}

\def\bM{\textbf{\textrm{M}}\hskip-0.03cm{}}

\def\bMr{\textbf{\textrm{\em M}}\hskip-0.03cm{}}
\def\bNr{\textbf{\textrm{\em N}}\hskip-0.03cm{}}

\def\bN{\textbf{\textrm{N}}\hskip-0.03cm{}}

\def\bT{\mathbf{T}}

\def\aA{\textrm A}

\DeclareMathOperator{\Ac}{\mathcal{A}}

\def\ac{c}

\def\ag{\textrm{g}}

\def\aw{\textrm{w}}

\DeclareMathOperator{\Cnt}{\textnormal{Cont}} 

\DeclareMathOperator{\eb}{\mathbf{e}}

\DeclareMathOperator{\Hc}{\mathcal{H}}
\DeclareMathOperator{\Hr}{\mathscr{H}}

\DeclareMathOperator{\Ic}{\mathcal{I}}

\DeclareMathOperator{\Lc}{\mathcal{L}}

\DeclareMathOperator{\mb}{{\textbf{m}}}
\DeclareMathOperator{\nb}{\mathbf{n}}

\DeclareMathOperator{\Rb}{\mathbb{R}}
\DeclareMathOperator{\Sb}{\mathbb{S}}

\newcommand{\spstar}{\textsf{null}} 
\newcommand{\spstarr}{{\textsf{\em null}}} 
\newcommand{\supp}{\textnormal{supp}}

\DeclareMathOperator{\gb}{\mathbf{g}}

\DeclareMathOperator{\hb}{\mathbf{h}}

\newcommand\hbw[1]{{\bT^{\<#1\>}\mathbf{{h}}}}

\DeclareMathOperator{\ub}{\mathbf{u}}
\DeclareMathOperator{\vb}{\mathbf{v}}

\newcommand\vbw[1]{\bT^{\<#1\>}\mathbf{{v}}}

\DeclareMathOperator{\wb}{\mathbf{w}}

\DeclareMathOperator{\xb}{\mathbf{x}}
\def\Xf{{\wh{X}}}

\DeclareMathOperator{\yb}{\mathbf{y}}
\DeclareMathOperator{\zb}{\mathbf{z}}
\DeclareMathOperator{\Zb}{\mathbb{Z}}

\DeclareMathOperator{\Ef}{\Theta} 
\DeclareMathOperator{\ef}{\text{\it e}} 
\DeclareMathOperator{\Qf}{ {\Gamma} } 
\DeclareMathOperator{\Vf}{\Omega} 
\DeclareMathOperator{\Vfm}{\Omega^0} 
\DeclareMathOperator{\Vfp}{\Omega^+} 

\DeclareMathOperator{\vf}{{\text{\it \sts v}}} 
\DeclareMathOperator{\vfs}{{\text{\it \sts v}}^\ast} 

\def\zero{\0}

\newcommand{\Mf}{\mathscr{M}}

\newcommand{\iA}{\textnormal{A}} 
\newcommand{\iB}{\textnormal{B}} 
\newcommand{\iF}{\textnormal{F}} 
\newcommand{\iJ}{\textnormal{J}} 
\newcommand{\iK}{\textnormal{K}} 
\newcommand{\iP}{\textnormal{P}} 
\newcommand{\ir}{\textnormal{r}} 
\newcommand{\iQ}{\textnormal{Q}} 
\newcommand{\iS}{\textnormal{S}} 
\newcommand{\iU}{\textnormal{U}} 
\newcommand{\iV}{\textnormal{V}} 
\newcommand{\iW}{\textnormal{W}} 

\newcommand{\Vol}{\textnormal{Vol}} 

\title{Introduction to the combinatorial atlas}
\date{\today}

\author[Swee Hong Chan \and Igor Pak]{Swee Hong Chan$^{\star}$ \ \. \and \ \. Igor~Pak$^{\star}$}

\thanks{\thinspace ${\hspace{-.45ex}}^\star$Department of Mathematics,
UCLA, Los Angeles, CA, 90095. \  Email: \. \texttt{\{sweehong,\ts{pak}\}@math.ucla.edu}}

\begin{document}
\begin{abstract}
We give elementary self-contained proofs of the \emph{strong Mason conjecture}
recently proved by Anari~et.~al~\cite{ALOV} and Br\"and\'en--Huh~\cite{BH}, and
of the classical \emph{Alexandrov--Fenchel inequality}.  Both proofs use the
\emph{combinatorial atlas} technology recently introduced by the authors~\cite{CP}.
We also give a formal relationship between combinatorial atlases and
\emph{Lorentzian polynomials}.
\end{abstract}

	\maketitle

\section{Introduction}

In this paper we tell three interrelated but largely independent stories.
While we realize that this sounds self-contradictory, we insist on this
description.
We prove no new results, nor do we claim to give new proofs of known results.
Instead, we give a \defng{new presentation} of the existing proofs.

Our goal is explain the \defna{combinatorial atlas} technology
from~\cite{CP} in three different contexts.  The idea is to both give
a more accessible introduction to our approach and connect it to
other approaches in the area.  Although one can use this paper as a
companion to~\cite{CP}, it is written completely independently and
aimed at a general audience.

\medskip

\nin
$(1)$ \ \defn{Strong Mason conjecture} claims ultra-log-concavity of the number
of independent sets of matroid according to its size.  This is perhaps the
most celebrated problem recently resolved in a series of paper culminating
with independent proofs by Anari~et.~al~\cite{ALOV} and
Br\"and\'en--Huh~\cite{BH}.  These proofs use the technology of
\emph{Lorentzian polynomials}, which in turn substantially simplify
earlier heavily algebraic tools.

In our paper~\cite{CP}, we introduce the \emph{combinatorial atlas} technology
motivated by geometric considerations of the \emph{Alexandrov--Fenchel inequality}.
This allowed us, among other things, to prove an advanced generalization of
the strong Mason conjecture to a large class of greedoids.  The conjecture
itself and its refinements followed easily from our more general results.
Our first story is a self-contained streamlined proof of just this conjecture,
without the delicate technical details necessary for our generalizations.

\medskip

\nin
$(2)$ \ \defn{Lorentzian polynomials} as a technology is an interesting
concept in its own right.  In~\cite{BH}, the authors showed not only how
to prove matroid inequalities, but also how to place the technology in the
context of ideas and approaches in other areas, including the above mentioned
Alexandrov--Fenchel inequality.

We do something similar in our second story, by showing that the theory of
Lorentzian polynomials is a special case of the theory of combinatorial atlases.
More precisely, for every Lorentzian polynomial we construct a combinatorial
atlas which mimics the polynomial properties and allows to derive the same
conclusions.

\medskip

\nin
$(3)$ \ The \defn{Alexandrov--Fenchel inequality} \ts is the classical geometric
inequality which remains deeply mysterious.  There are several algebraic
and analytic proofs, all of them involved and technical, to different degree.
Much of the combinatorial atlas technology owes to our deconstruction of the
insightful recent proof in~\cite{SvH19} by Shenfeld and van Handel.

In the third story, we proceed in the reverse direction, and prove the Alexandrov--Fenchel
inequality by the tools of the combinatorial atlas.  The resulting proof
is similar to that in~\cite{SvH19}, but written in a different language and
filling the details not included in~\cite{SvH19}.  Arguably, this is the first 
exposition of the proof of the Alexandrov--Fenchel inequality that is 
both elementary and self-contained.

\medskip

The paper structure is very straightforward.  After the short
notation section (Section~\ref{s:def}), we define the combinatorial atlas
and state its properties (Section~\ref{s:atlas}).  This is a prequel to
all three sections that follow, all of which are independent from each
other, and cover items~$(1)$, $(2)$ and~$(3)$.  At risk of
repeating ourselves, let us emphasize that these three
Sections~\ref{s:matroid},~\ref{s:Lorentzian} and~\ref{s:AF ineq}
can be read in any order.  We conclude with brief final remarks
in Section~\ref{s:finrem}.  For further background and
historical remarks, see the extensive $\S$16 and~$\S$17 in~\cite{CP}.

\bigskip

\section{Definitions and notations}\label{s:def}

We use \ts $[n]=\{1,\ldots,n\}$, \ts $\nn = \{0,1,2,\ldots\}$,
\ts $\Zb_+ = \{1,2,\ldots\}$, \ts $\rrp=\{x\ge 0\}$ \ts and \ts $\rrs=\{x> 0\}$.
For a subset \ts $S\subseteq X$ \ts and element \ts $x\in X$,
we write \ts $S+x:=S\cup \{x\}$ \ts and \ts $S-x:=S\sm\{x\}$.

\smallskip

Throughout the paper we denote matrices with bold capitalized letter and
the entries by roman capitalized letters: \ts $\bM=(\aM_{ij})$.  We also keep
conventional index notations, so, e.g., \. $\big( \bM^3+ \bM^2 \big)_{ij}$ \. is the $(i,j)$-th
matrix entry of \. $\bM^3+\bM^2$.
We denote vectors by bold small letters, while vector entries by either unbolded uncapitalized letters
or vector components, e.g.\ \ts $\hb = (\ah_1,\ah_2,\ldots)$ \ts and \ts $\ah_i= (\hb)_i$.

A real matrix (respectively, a real vector) is \defn{nonnegative} if all its entries are nonnegative real numbers, and
is \defn{strictly positive} if all of its entries are positive real numbers.
The \defn{support} of a real \ts $\ar \times \ar$ \ts symmetric matrix $\bM$ is defined as:
 \[
 \supp\ts (\bM) \ := \ \bigl\{ \.  i\in [\ar] \. : \. \aM_{ij} \neq 0  \ \text{ for some } j \in [\ar] \. \bigr\}.
 \]
In other words,  $\supp(\bM)$ \ts is  the set of indexes for which the corresponding row and
column of $\bM$ are nonzero vectors.
Similarly, the \defn{support} of a real $\ar$-dimensional vector $\hb$ is defined as:
    \[  \supp\ts (\hb) \ := \ \{ \.  i \in [\ar] \. : \. \ah_{i} \neq 0  \. \}. \]
For  vectors \ts $\vb, \wb \in \Rb^{\ar}$, we write \ts $\vb \leqslant \wb$ \ts
to mean the componentwise inequality, i.e.\ \ts $\av_i \leq \aw_i$ \ts for all $i \in [\ar]$.
We write \ts $|\vb| := \av_1+\ldots+\av_{\ar}$.  We also use \ts
$\eb_1,\ldots, \eb_{\ar}$ \ts to denote the standard basis of \ts $\Rb^{\ar}$.

For a subset \ts $S\subseteq [\ar]$,
   the \defn{characteristic vector} of~$S$  is the vector \. $\vb \in \Rb^{\ar}$ \. such that
  \ts $\av_i=1$ \ts if \ts $i \in S$ and \ts $\av_i=0$ \ts if \ts $i \notin S$.
   We use \ts $\0 \in \Rb^{\ar}$ \ts to denote the zero vector.
Denote by \ts $\Sb^{n-1}$ \ts the set of unit vectors in \ts
$\Rb^n$, i.e.\ vectors \ts $\vb\in\rr^\ar$ \ts
with Euclidean norm \ts $\|\vb\|=1$.  We use \. $\Vol_n(P)$ \. to denote $n$-dimensional
volume of polytope~$P$.  When \ts $n=2$, we write \. $\area(P)=\Vol_2(P)$ \. for the
area of polygon \ts $P\ssu \rr^2$.
We adopt the convention that \. $\Vol_0(P)=1$ \.  when \ts $P$
\ts is a point.

Finally, we make a frequent use of (lesser known) trigonometric functions
$$\csc \theta \. := \. \frac{1}{\sin \theta} \qquad \text{and} \qquad
\cot \theta \. := \. \frac{\cos \theta}{\sin \theta}\,.$$

\bigskip

\section{Combinatorial atlases and hyperbolic matrices}
\label{s:atlas}

In this section we introduce combinatorial atlases and present
the local–global principle which allows one to recursively
establish hyperbolicity of vertices.

\subsection{Combinatorial atlas}\label{ss:atlas-def}
Let \ts $\cP=(\Vf,\prec)$ \ts be a locally finite poset of bounded height.\footnote{In our
examples, the poset~$\cP$ can be both finite and infinite.}
Denote by \. $\Qf=(\Vf, \Ef):=\Hc_\cP$ \. be the acyclic digraph given by the
Hasse diagram $\Hc_\cP$ of~$\cP$.  Let \ts $\Vfm\subseteq\Vf$ \ts be the set of maximal
elements in~$\cP$, so these are \defn{sink vertices} in~$\Qf$.  Similarly, denote by
\ts $\Vfp:=\Vf\sm \Vfm$ \ts the \defn{non-sink vertices}.  We write \ts $\vfs$ \ts
for the set of out-neighbor vertices \ts $v'\in \Vf$, such that \ts $(v,v')\in \Ef$.

\smallskip

\begin{definition}{\rm
A \defna{combinatorial atlas} \ts $\AA=\AA_\cP$ \ts of dimension \ts $\ar$ \ts
is an acyclic digraph   \. $\Qf:=(\Vf, \Ef)$ \. with an additional structure:

\smallskip

$\circ$ \ Each vertex \ts $\vf \in \Vf$ \ts is associated with a  pair \ts $(\bM_{\vf}, \hb_{\vf})$\ts, where \ts
$\bM_{\vf}$ \ts is a  symmetric \ts  $\ar \times \ar$ \ts matrix

\quad \. with nonnegative diagonals,
and \ts $\hb_{\vf}\in \rr_{\ge 0}^\ar$ \ts is a nonnegative vector.

\smallskip
	
$\circ$ \ The outgoing edges of each vertex \ts $\vf \in \Vfp$ \ts
are labeled with indices \ts $i\in [\ar]$, without repetition.

\quad \.
We denote the edge labeled \ts $i$ \ts as  \ts
$\ef^{\<i\>}=(\vf,\vf^{\<i\>})$, where \ts $1 \leq i \leq \ar$.

\smallskip
	
$\circ$ \ Each edge $\ef^{\<i\>}$ is associated to a linear transformation \.
$\bT^{\langle i \rangle}_{\vf}: \ts \Rb^{\ar} \to \Rb^{\ar}$.

\smallskip

\nin
Whenever clear, we drop the subscript~$\vf$ to avoid cluttering.
We call \. $\bM=(\aM_{ij})_{i,j \in [\ar]}$ \. the \defn{associated matrix} of $\vf$,
and \. $\hb = (\ah_i)_{i \in [\ar]}$  \. the \defn{associated vector} of~$\vf$.
In notation above, we have \ts $\vf^{\<i\>}\in \vfs$, for all \ts $1\le i \le \ar$.

}
\end{definition}

\begin{rem}
Note that in \cite{CP}, the matrix $\bM_v$ is a nonnegative matrix.
We use a weaker condition here so that we can prove Alexandrov--Fenchel
inequality, cf.~\cite[$\S$17.6]{CP}. \end{rem}

\smallskip

\subsection{Local-global principle}\label{ss:atlas-local}
A matrix $\bM$ is called \defn{hyperbolic}, if
	\begin{equation}\label{eq:Hyp}
\langle \vb, \bM \wb \rangle^2 \  \geq \  \langle \vb, \bM \vb \rangle  \langle \wb, \bM \wb \rangle \quad \text{for every \ \, $\vb, \wb \in \Rb^{\ar}$, \ \  such that \ \ $\langle \wb, \bM \wb \rangle > 0$}. \tag{Hyp}
\end{equation}
For the atlas \ts $\AA$, we say that \ts $\vf\in \Vf$ \ts is \defn{hyperbolic},
if the associated matrix \ts $\bM_{\vf}$ \ts is hyperbolic, i.e.\ satisfies \eqref{eq:Hyp}.
We say that atlas \ts $\AA$ \ts satisfies \ts\defng{hyperbolic property} \ts if every
 \ts $\vf\in \Vf$ \ts is hyperbolic.

\smallskip

Note that property \eqref{eq:Hyp} depends only on the support of $\bM$,
i.e.\ it continues to hold after adding or removing zero rows or columns.
This simple observation will be used repeatedly through the paper.

\smallskip

We say that atlas \ts $\AA$ \ts satisfies \ts\defng{inheritance property} \ts if for
every non-sink vertex \ts $\vf\in \Vfp$, we have:
\begin{equation}\label{eq:Inh}\tag{Inh}
\begin{split}
		&	  (\bM \vb )_{i} \ = \ \big \langle   \bT^{\<i\>}\vb, \, \bM^{\<i\>} \,
\bT^{\<i\>}\hb \big \rangle \quad \text{ for every} \ \ \, i \in \supp(\bM) \text{ \ \ and \ \. }  \vb \in \Rb^{\ar},
\end{split}
\end{equation}
where \ts \ts $\bT^{\<i\>}=\bT^{\<i\>}_{\vf}$\., \ts $\hb=\hb_{\vf}$ \. and
\. $\bM^{\<i\>}:=\bM_{{\vf}^{\<i\>}}$ \. is the matrix associated with~$\vf^{\<i\>}$\..

\smallskip

Similarly, we say that atlas \ts $\AA$ \ts satisfies \ts\defng{pullback property} \ts if for
every non-sink vertex \ts $\vf\in \Vfp$, we have:
\begin{equation}\label{eq:Pull}\tag{Pull}
		\sum_{i\ts\in\ts \supp(\bM)} \. \ah_{i} \, \big \langle   \bT^{\<i\>}\vb, \, \bM^{\<i\>} \,  \bT^{\<i\>}\vb  \big \rangle  \ \geq \ \langle \vb, \bM \vb \rangle \qquad \text{ for every } \vb \in \Rb^{\ar},
\end{equation}
and we say that atlas \ts $\AA$ \ts satisfies \ts\defng{pullback equality property} \ts if for
every non-sink vertex \ts $\vf\in \Vfp$, we have:
\begin{equation}\label{eq:PullEq}\tag{PullEq}
	\sum_{i\ts\in\ts \supp(\bM)} \. \ah_{i} \, \big \langle   \bT^{\<i\>}\vb, \, \bM^{\<i\>} \,  \bT^{\<i\>}\vb  \big \rangle  \ = \ \langle \vb, \bM \vb \rangle \qquad \text{ for every } \vb \in \Rb^{\ar}.
\end{equation}
Clearly \eqref{eq:PullEq}  implies \eqref{eq:Pull}.
All log-concave inequalities in this paper satisfy this stronger property~\eqref{eq:PullEq};
we refer to \cite{CP} for applications of~\eqref{eq:Pull} when \eqref{eq:PullEq} is not satisfied.

\smallskip

We say that a non-sink vertex \ts $\vf\in \Vfp$ \ts is \defn{regular} if the following positivity conditions are satisfied:
\begin{align}\label{eq:Irr}\tag{Irr}
& \text{The associated matrix \ts $\bM_{\vf}$ \ts restricted to its support is irreducible.} \\
\label{eq:hPos} \tag{h-Pos} &\text{Vectors  \ts $\hb_{\vf}$ \ts  and \ts $\bM_{\vf}\hb_{\vf}$ \ts  are strictly positive when restricted to the support of \ts $\bM_{\vf}$\..}
\end{align}

\smallskip

\begin{rem}
In \cite{CP}, \eqref{eq:hPos} does not impose positivity on \ts $\bM_{\vf} \hb_{\vf}$, since in that setting this vector is positive by the positivity of \ts $\hb_{\vf}$ \ts and non-negativity of \ts $\bM_{\vf}$.  Note also that \eqref{eq:PullEq} is a new property not mentioned in~\cite{CP}.
\end{rem}

\smallskip

\begin{thm}[{\rm {\em \defna{local–global principle}, {\em see}~\cite[Thm~5.2]{CP}}}{}]
\label{t:Hyp}
	Let $\AA$ be a combinatorial atlas that satisfies properties \eqref{eq:Inh} and \eqref{eq:Pull}, and
	let \ts $\vf\in \Vfp$ \ts be a non-sink regular vertex of~$\Qf$.
	Suppose every out-neighbor of \ts $\vf$ \ts is hyperbolic.
	Then \ts $\vf$ \ts is also hyperbolic.
\end{thm}

\smallskip

Theorem~\ref{t:Hyp} reduces checking the property \eqref{eq:Hyp} to sink vertices \ts $v\in \Vfm$.
In our applications, the pullback property \eqref{eq:PullEq} is more involved than the inheritance
property~\eqref{eq:Inh}.  Below, in  Theorem~\ref{t:Pull}, we give sufficient conditions for
\eqref{eq:PullEq} that are easier to establish.

\smallskip

\subsection{Eigenvalue interpretation of hyperbolicity}
The following lemma  gives two equivalent conditions to~\eqref{eq:Hyp}
that is often easier to check.
A symmetric matrix $\bM$ satisfies \eqref{eq:NDC} \ts if
\begin{equation}\label{eq:NDC}
	\tag{NDC} \text{There exists \. $\gb  \in \Rb^{\ar}$ \ s.t.\ \qquad  $\forall$ \ts
$\vb \in \Rb^{\ar}$, \. $\< \vb, \bM \gb \> = 0$ \qquad $\Longrightarrow$ \qquad $\< \vb, \bM \vb \>\leq 0$. }
\end{equation}
Here~\eqref{eq:NDC} stands for \defng{negative semi-definite in the complement}.
This condition does not appear in~\cite{CP}, and is needed here
for a step in Alexandrov--Fenchel inequality.

A symmetric matrix $\bM$ satisfies \eqref{eq:OPE} if
\begin{equation}\label{eq:OPE}\tag{OPE}
	\text{$\bM$ \. has at most \. \defng{one positive eigenvalue} \. (counting multiplicity).}
\end{equation}
The equivalence between these three properties are well-known in the literature, see e.g.\
\cite{Gre}, \cite[Thm~5.3]{COSW}, \cite[Lem.~2.9]{SvH19}  and \cite[Lem.~2.5]{BH}.
We present a short proof for completeness; we follow \cite[Lem.~5.3]{CP} in our presentation.

\smallskip

	\begin{lemma}\label{l:Hyp is OPE}
	Let \ts {\rm $\bM$} \ts be a self-adjoint  operator on $\Rb^{\ar}$ for an inner product $\langle \cdot, \cdot \rangle$.
	Then:
	\[  	\text{\ts {\rm $\bM$} \ts satisfies \eqref{eq:Hyp} \ts} \ \Longleftrightarrow  \  \text{{\rm $\bM$} \ts satisfies \eqref{eq:NDC}}
	\ \Longleftrightarrow  \  \text{{\rm $\bM$} \ts satisfies \eqref{eq:OPE}}. \]
\end{lemma}

\smallskip

\begin{proof}
	If $\bM$ is a negative semidefinite matrix, then the conclusion is trivial.
	Thus we assume that $\bM$ has a positive eigenvalue $\lambda_1 >0$, which we assume to be the largest eigenvalue.
	
For the \ts  \eqref{eq:OPE} \ts $\Rightarrow$ \ts \eqref{eq:NDC} \ts direction,
let \ts $\gb$ \ts be an eigenvector of~$\lambda_1$.
Note that \. $\<\gb, \bM \gb\>  = \lambda_1 \<\gb,\gb\> >0$.
	Then, for every \. $\vb \in \Rb^{\ar}$ \. such that
\. $\<\vb, \bM \gb\>  = 0$, we have
\[ \<\vb, \bM \vb\>  \  \leq \ \lambda_2 \sqrt{|\<\vb,  \vb\>|},\]
	where $\lambda_2$ is the second largest eigenvalue of $\bM$.
	Note that $\lambda_2 \leq 0$ by \eqref{eq:OPE}, and it then follows that the right side of the equation above is non-negative.
	This proves \eqref{eq:NDC}.

We now prove \ts  \eqref{eq:NDC} \ts $\Rightarrow$ \ts \eqref{eq:Hyp} \ts direction.
Since \. $\langle \wb, \bM \wb \rangle > 0$,
it then follows from
  \eqref{eq:NDC} that \. $\langle \wb, \bM \gb \rangle   \neq   0$.
  Let \ts $\zb \in \Rb^{\ar}$ \ts be the vector
  \[ \zb \ := \  \vb  \ - \ \frac{\langle \vb, \bM \gb \rangle}{\langle \wb, \bM \gb \rangle} \, \wb.\]
It follows that $\langle \zb, \bM \gb \rangle   =0 $.
By \eqref{eq:NDC}, this implies that
	 \. $\langle \zb, \bM \zb \rangle  \leq  0$.
Now note that
\begin{align*}
0 \ \geq \ 	\langle \zb, \bM \zb \rangle \ & =  \ \langle \vb, \bM \vb \rangle  \ - \ 2 \. \frac{\langle \vb, \bM \gb \rangle \, \langle \vb, \bM \wb \rangle}{\langle \wb, \bM \gb \rangle}  \ + \  \frac{\langle \vb, \bM \gb \rangle^2 \, \langle \wb, \bM \wb \rangle}{\langle \wb, \bM \gb \rangle^2} \\
	& \geq \ \langle \vb, \bM \vb \rangle \ - \  \frac{\langle \vb, \bM \wb \rangle^2}{\langle \wb, \bM \wb \rangle}\.,
\end{align*}	
where the last inequality is due to the AM--GM inequality.
This proves \eqref{eq:Hyp}, as desired.

For the \ts  \eqref{eq:Hyp} \ts $\Rightarrow$ \ts \eqref{eq:OPE} \ts direction,
suppose to the contrary that $\bM$ has eigenvalues \ts $\lambda_1, \lambda_2>0$ (not necessarily distinct).
Let $\vb$ and $\wb$ be  orthonormal eigenvectors of $\bM$  for $\lambda_1$ and $\lambda_2$, respectively.	
It then follows that
\begin{align*}
	0  \ = \ \langle \vb, \bM \wb \rangle \quad \text{ and } \quad \langle \vb, \bM \vb \rangle \, \langle \wb, \bM \wb \rangle \ = \  \lambda_1  \lambda_2\.,
\end{align*}
which contradicts \eqref{eq:Hyp}.
\end{proof}

\medskip

\subsection{Proof of Theorem~\ref{t:Hyp}}
Let \ts $\bM:=\bM_{\vf}$ \ts and \ts $\hb:=\hb_{\vf}$ \ts be the associated matrix and the associated vector of $\vf$, respectively.
	Since \eqref{eq:Hyp} is a property that is invariant under restricting to the support of $\bM$,
	it follows from \eqref{eq:Irr} that we can assume that $\bM$ is irreducible.
	
	Let  \ts $\bD:= (\aD_{ij})$  \ts be the \ts $\ar \times \ar$ \ts diagonal matrix given by
\[   \aD_{ii} \ := \   \frac{(\bM \hb)_{i}}{\ah_i} \qquad \text{ for every } \ 1\le i \le \ar\ts.
\]
	Note that $\bD$ is well defined and  	\. $\aD_{i i} >0$,  by \eqref{eq:hPos} and the assumption that $\bM$ is irreducible.
	Define a new inner product \ts $\langle \cdot, \cdot \rangle_{\bD}$ \ts on \ts $\Rb^{\ar}$ \ts by
	\.
	$ \langle \vb, \wb \rangle_{\bD} \ := \  \langle \vb, \bD \wb\rangle$\..
	
	Let \. $\bN := \bD^{-1} \bM$\ts.
	Note that 	\.
	$ \langle \vb, \bN \wb \rangle_{\bD}  =   \langle \vb, \bM \wb\rangle$ \. for every $\vb,\wb \in \Rb^{\ar}$.
	Since $\bM$ is a symmetric matrix,
	this implies that  $\bN$ is  a self-adjoint operator on $\Rb^{\ar}$ for the inner product $\langle \cdot, \cdot \rangle_{\bD}$\..
	A direct calculation shows that	 \ts $\hb$ \ts is an eigenvector of \ts $\bN$ \ts for eigenvalue \ts $\lambda=1$.
	Since $\bM$ is irreducible matrix and $\hb$ is a strictly positive vector,
	it then follows from the  Perron--Frobenius theorem that \ts $\lambda=1$ \ts is the
    largest real eigenvalue of $\bN$, and that it has multiplicity one.	
	
\medskip

\nin
\textbf{Claim:} \ {\em $\lambda=1$ \. is the only positive eigenvalue of \ts $\bNr$ \ts $($counting multiplicity$)$.}

\medskip

	Applying Lemma~\ref{l:Hyp is OPE} to the matrix $\bN$ and the inner product $\langle \cdot, \cdot \rangle_{\bD}$\., we have:
	\[ \langle \vb, \bN \wb \rangle_{\bD}^2 \  \geq \  \langle \vb, \bN \vb \rangle_{\bD} \  \langle \wb, \bN \wb \rangle_{\bD} \quad   \text{ for every } \vb, \wb \in \Rb^{\ar}, \ \  \text{such that \ \ $\langle \wb, \bM \wb \rangle > 0$}. \]
	Since 	\.
	$ \langle \vb, \bN \wb \rangle_{\bD}  =   \langle \vb, \bN \wb\rangle$, this implies \eqref{eq:Hyp} for~$\vf$, and completes the proof
of the theorem.
\qed

\medskip

\begin{proof}[Proof of the Claim]		
		Let $i \in [\ar]$
		 and $\vb \in \Rb^{\ar}$.		
		It follows from \eqref{eq:Inh} that
			\begin{align}\label{eq:general 0}
			\big((\bM \vb)_{i}\big)^2 \ = \  \big \langle    \vbw{i}, \, \bM^{\<i\>} \, \hbw{i}  \big \rangle^2.
		\end{align}
		Since \ts $\bM^{\<i\>}$ \ts satisfies \eqref{eq:Hyp} by the assumption of the theorem, applying \eqref{eq:Hyp} to the RHS of \eqref{eq:general 0} gives:
		\begin{align}\label{eq:general 1}
			\big((\bM \vb)_{i}\big)^2 \   \geq \  \big \langle   \vbw{i}, \, \bM^{\<i\>} \, \vbw{i}  \big \rangle  \,  \big \langle   \hbw{i}, \, \bM^{\<i\>} \, \hbw{i}  \big \rangle,
		\end{align}
Here \eqref{eq:Hyp} can be applied since
		\. $\big \langle   \hbw{i}, \, \bM^{\<i\>} \, \hbw{i}  \big \rangle \. = \. (\bM \hb)_{i} >0 $.
		Now note that
		\begin{align*}
			\big((\bN \vb)_{i}\big)^2 \, \aD_{i i} \ &= \  \big((\bM \vb)_{i}\big)^2  \, \frac{\ah_{i}}{(\bM \hb)_{i}} \ =_{\eqref{eq:Inh}}  \ \big((\bM \vb)_{i}\big)^2  \,  \frac{\ah_{i}}{\big \langle   \hbw{i}, \, \bM^{\<i\>} \, \hbw{i}  \big \rangle } \\
			& \geq_{\eqref{eq:general 1}} \ \ah_{i} \, \big \langle   \vbw{i}, \, \bM^{\<i\>} \, \vbw{i}  \big \rangle.
		\end{align*}
		Summing this inequality over all \ts $i \in [\ar]$, gives:
		\begin{align}\label{eq:general 2}
			\langle \bN \vb, \bN \vb \rangle_{\bD} \ \geq \  \sum_{i=1}^r  \ah_{i} \,   \big \langle   \vbw{i}, \, \bM^{\<i\>} \, \vbw{i}  \big \rangle \ \geq_{\eqref{eq:Pull}}
			\ \langle \vb, \bM \vb \rangle  \ = \  \langle  \vb, \bN \vb \rangle_{\bD}\..
		\end{align}
		Now, let $\lambda$ be an arbitrary eigenvalue of $\bN$,
		and let
		$\gb$ be an eigenvector of  $\lambda$.
		We have:
		\begin{align*}
			\lambda^2  \langle  \gb,  \gb \rangle_{\bD} \ = \  	\langle \bN \gb, \bN \gb \rangle_{\bD}
			\ \geq_{\eqref{eq:general 2}} \ 	 \langle  \gb, \bN \gb \rangle_{\bD}  \ = \ \lambda \. \langle  \gb,  \gb \rangle_{\bD}\..	\end{align*}
		This implies that $\lambda\geq 1$ or $\lambda \leq 0$.
		Since $\lambda =1$ is the largest eigenvalue of \ts $\bN$ \ts and is simple, we obtain the result.
	\end{proof}

\smallskip

\begin{rem}\label{r:Hyp-claim}
In the proof above, neither the Claim nor the proof of the Claim are new,
but a minor revision of Theorem 5.2 in~\cite{SvH19}.  We include the proof
for completeness and to help the reader get through our somewhat cumbersome
notation.  \end{rem}

\medskip

\subsection{Pullback equality property}\label{ss:Pull}
Here we present a sufficient condition for \eqref{eq:PullEq} that is easier to verify.
This condition is a more restrictive version of the sufficient conditions for \eqref{eq:Pull} in \cite[\S6]{CP}.
We also remark that this condition applies to atlases in Sections~\ref{s:matroid} and~\ref{s:Lorentzian},
but does not apply to atlases in Section~\ref{s:AF ineq}.

Let $\AA$ be a combinatorial atlas.  We say that $\AA$ satisfies the
\defng{identity property}, if for every non-sink vertex \ts
$\vf\in \Vfp$ \ts and every \ts $i \in \supp(\bM)$, we have:
\begin{equation}\label{eq:Iden}\tag{Iden}
	\bT^{\<i\>}: \Rb^{\ar} \to \Rb^{\ar} \ \text{ is the identity mapping.}
\end{equation}	
We say that $\AA$ satisfies the \defng{transposition-invariant property},  if for every non-sink vertex \ts
$\vf\in \Vfp$,
\begin{equation}\label{eq:TPInv}\tag{T-Inv}
	\aM_{jk}^{\<i\>}	\ = \  \aM_{ki}^{\<j\>}	 \ = \ \aM_{ij}^{\<k\>} \qquad 	\text{ for every  $i, j, k \in \supp(\bM)$.}
\end{equation}

We say that $\AA$ has the \defng{decreasing support property}, if for every non-sink vertex \ts $\vf \in \Vfp$,
\begin{equation}\label{eq:Decsupp}
	\tag{DecSupp}
	\supp(\bM) \ \supseteq \ \supp\big(\bM^{\<i\>}\big) \quad \text{ for every $i \in \supp(\bM)$.}
\end{equation}

\smallskip

\begin{rem}
Note that there is a small difference from \eqref{eq:TPInv} in \cite[$\S$6.1]{CP}, namely that in~\cite{CP} the condition only applies to \emph{distinct} \ts $i,j,k$.  Note also that  \eqref{eq:Decsupp} is a new property, that was not defined in~\cite{CP}.
\end{rem}

\smallskip

\begin{thm}[{\rm cf.~\cite[Thm~6.1]{CP}}{}]
\label{t:Pull}
	Let \ts $\AA$ \ts be a combinatorial atlas that satisfies \eqref{eq:Inh}, \eqref{eq:Iden}, \eqref{eq:TPInv} and \eqref{eq:Decsupp}.
	Then  \ts $\AA$ \ts also satisfies  \eqref{eq:PullEq}.
\end{thm}

\smallskip

\begin{proof}
	Let $\vf$ be a non-sink vertex of $\Qf$,
	and let $\vb \in \Rb^{\ar}$.
	The LHS of \eqref{eq:PullEq} is equal to
	\begin{align*}
		\sum_{i \ts\in\ts \supp(\bM)} \ah_{i} \, \big\langle \vbw{i}, \bM^{\<i\>} \vbw{i} \big\rangle  \
		= \ \sum_{i \ts\in\ts \supp(\bM)} \ \sum_{j, \ts k \ts\in\ts \supp(\bM^{\<i\>}) } \ah_{i} \, \big(\vbw{i}\big)_j \, \big(\vbw{i}\big)_k \, \aM^{\<i\>}_{jk}\..
	\end{align*}
	By \eqref{eq:Iden} and \eqref{eq:Decsupp}, this gives:
	\begin{align}\label{eq:families left}
		\sum_{i \ts\in\ts \supp(\bM)} \ah_{i} \, \big\langle \vbw{i}, \bM^{\<i\>} \vbw{i} \big\rangle  \
		= \ \sum_{i, \ts j, \ts k \ts\in\ts \supp(\bM)}  \ah_{i} \, \av_j \, \av_k \, \aM^{\<i\>}_{jk}\..
	\end{align}
	
	On the other hand, the RHS of \eqref{eq:PullEq} is equal to
	\begin{align*}
		\langle \vb, \bM \vb \rangle  \ &= \  \sum_{i' \ts\in\ts \supp(\bM)} \av_{i'} \, (\bM \vb)_{i'}  \ =_{\eqref{eq:Inh}} \  \sum_{i' \ts\in\ts \supp(\bM)} \av_{i'} \, \big \langle   \vbw{i'}, \, \bM^{\<i'\>} \,   \hbw{i'}  \big \rangle  \notag \\
	\ &= \ \sum_{i' \ts\in\ts \supp(\bM)} \ \sum_{j', \ts k' \ts\in\ts\supp\left(\bM^{\<i'\>}\right)}  \av_{i'} \, \big(\vbw{i'}\big)_{j'} \, \big(\hbw{i'}\big)_{k'} \, \aM^{\<i'\>}_{j'k'}\..
	\end{align*}
	By \eqref{eq:Iden} and \eqref{eq:Decsupp}, this gives:
	\begin{align}\label{eq:families right}
		\langle \vb, \bM \vb \rangle  \ = \ \sum_{i',j',k' \ts\in\ts \supp(\bM)}  \av_{i'} \, \av_{j'} \, \ah_{k'} \, \aM^{\<i'\>}_{j'k'}\..
	\end{align}
	Let us show that each term in the RHS of~\eqref{eq:families left} is equal to that of RHS of~\eqref{eq:families right}
after the substitution \. $i' \gets j$, $j' \gets k$, $k' \gets i$.  Indeed, we have:
	\[  \ah_i \, \av_j  \, \av_k  \, \aM^{\<i\>}_{jk} \ =_{\eqref{eq:TPInv}} \  \ah_i \, \av_j  \, \av_k  \, \aM^{\<j\>}_{ki}  \ = \   \av_{i'} \, \av_{j'}  \, \ah_{k'} \, \aM^{\<i'\>}_{j'k'}\..  \]
	This implies that the LHS of~\eqref{eq:families left} is equal to the LHS of~\eqref{eq:families right}, as desired.
\end{proof}

\bigskip

\section{Log-concavity for matroids}\label{s:matroid}

\subsection{Log-concavity of independent sets}
A (finite) \defna{matroid}~$\Mf$ \ts is a pair \ts $(X,\Ic)$ \ts of a
\defnb{ground set} \ts $X$, \ts $|X|=n$, and a nonempty collection of \defnb{independent sets}
\ts $\Ic \subseteq 2^X$ \ts that satisfies the following:
\begin{itemize}
	\item (\defng{hereditary property}) \, $S\subset T$, \. $T\in \Ic$ \, $\Rightarrow$ \, $S\in \Ic$\ts, and
	\item (\defng{exchange property}) \,  $S, \ts T\in \Ic$, \.  $|S|<|T|$ \, $\Rightarrow$ \,
	$\exists \ts x \in T \setminus S$ \. s.t.\ $S+x \in \Ic$\ts.
\end{itemize}
\defnb{Rank} of a matroid is the maximal size of the independent set: \ts $\rk(\Mf) := \max_{S\in \Ic} \ts |S|$.
A \defnb{basis} of a matroid is an independent set of size \ts $\rk(\Mf)$.  Finally,
let \ts $\cI_k \ts := \ts \bigl\{S\in \cI, \. |S|=k\bigr\}$,
and let \ts $\rI(k)=\bigl|\cI_k\bigr|$ \ts be the \defn{number of independent sets}
in \ts $\Mf$ \ts of size~$k$, \ts $0\le k \le \rk(\Mf)$.

\smallskip

We assume the reader is familiar with basic ideas of matroids, even though we
will not be using any properties other than the definitions.  The reader
unfamiliar with matroids can always assume that the matroid~$\Mf$ is given by
a set of vectors \ts $X \in \mathbb{K}^d$, with linearly independent subsets \.
$S \subseteq X$ \. being independent sets of the matroid: \ts $S\in \Ic$.

\smallskip

In this section we give a new proof of the ultra-log-concavity conjecture of Mason~\cite{Mas}.
We start with a weaker version below.

\smallskip

\begin{thm}[{\rm  \defng{Log-concavity for matroids}, \cite[Cor.~9]{HSW}, formerly \defna{weak Mason conjecture}}{}]
	\label{t:matroids-HSW}
	For a matroid \ts $\Mf=(X,\Ic)$,  \ts $|X|=n$,  \ts and integer \ts $1\le k < \rk(\Mf)$, we have:
	\begin{equation}\label{eq:matroid-ULC}
		\rI(k)^2 \, \ge \, \left(1 \. + \, \frac{1}{k}\right)  \,\.  \rI(k-1)  \   \rI(k+1)\ts.
	\end{equation}
\end{thm}

\smallskip

This result was recently proved by Huh, Schr\"oter and Wang in~\cite{HSW} using the
\emph{Hodge theory for matroids}.  Note that a slightly weaker but historically
first log-concavity inequality
$$\rI(k)^2 \, \ge \,\.  \rI(k-1)  \   \rI(k+1)
$$
in the generality of all matroids was proved by Adiprasito, Huh and Katz
in~\cite[Thm~9.9~(3)]{AHK}.
In the rest of this section we prove Theorem~\ref{t:matroids-HSW} and its extension
Theorem~\ref{t:matroids-BH} by using the combinatorial atlas theory.

\medskip

\subsection{Matroids as languages}
\label{ss:matroid-languages}
Let \ts $\Mf=(X,\Ic)$ \ts be a matroid of rank \ts $\rk(\Mf)$.
Let \. $\alpha=x_1\ldots x_\ell \in X^*$ \. be a word in the alphabet~$X$,
where \ts $X^\ast$ \ts is the set of finite words in the alphabet~$X$.
We say that $\alpha$ is \defn{simple}  if all letters occur at most once.
Denote by \. $|\alpha|:=\ell$ \. the length of~$\alpha$.

Word \ts $\alpha$ \ts is called \defnb{feasible} if \.
$\alpha$ \ts is simple and \.
$\{x_1,\ldots, x_\ell\} \in \Ic$.
We denote by \ts $\cL$ \ts the set of feasible words of $\Mf$, and by \ts
$\cL_k$ \ts the set of feasible words of length $k\geq 0$, where \. $0 \le k \le \rk(\Mf)$.
Note that \ts $\cL$ \ts satisfies the following properties:
\begin{itemize}
	\item  (\defng{hereditary property}) \, $\al\be\in \cL$  \ $\Rightarrow$ \ $\al\in \cL$, and
	\item  (\defng{exchange property}) \,  $\al, \be\in \cL$ \. s.t.\  \. $|\alpha| > |\beta|$
	\ $\Rightarrow$ \ $\exists\ts x\in \al$ \, s.t.\  \. $\beta x \in \Lc$\ts,
	\item (\defng{matroid symmetry propery}) \,
	$\alpha=x_1\ldots x_\ell \in \cL$  \ $\Rightarrow$ \ $x_{\sigma(1)} \ldots x_{\sigma(\ell)} \in \Lc$ \.   $\forall \. \sigma \in S_\ell$.
\end{itemize}
Here we write \ts $x \in \alpha$ \ts if the letter $x$ occurs in the word~$\alpha$.
Let us mention that the first two properties imply that $\cL$ is the language set of a \defng{greedoid},
see e.g.\ \cite{BZ,KLS}.

\smallskip

For every \ts $\alpha=\al_1\ldots \al_\ell \in X^*$, the set of \defnb{continuations} of $\alpha$ is defined as
\begin{equation*}
	\Cnt(\alpha) \, := \. \bigl\{x \in X ~\mid~ \alpha x \in \Lc \bigr\}.
\end{equation*}
In particular,  \. $\Cnt(\alpha) \in X \setminus \{\al_1,\ldots, \al_\ell\}$ \. and that \. $\Cnt(\alpha) \neq \varnothing$ \. only if \. $\alpha \in \cL$.
More generally, for \ts $k \geq 1$, we write
\begin{equation*}
	\Cnt_k(\alpha) \, := \. \bigl\{ \. \beta \in X^* ~\mid~ \alpha \beta \in \Lc  \. \text{ and } |\beta|=k\. \bigr\},
\end{equation*}
and note that \. $\Cnt(\alpha)=\Cnt_1(\alpha)$.

\medskip

\subsection{Combinatorial atlas for matroids}
\label{ss:Combinatorial-atlas-matroid}
Let \ts $\Mf=(X,\Ic)$ \ts be a matroid, and
let \ts $1\le k < \rk(\Mf)$.
We define a combinatorial atlas \ts $\AA$ \ts corresponding
to \ts $(\Mf,k)$ \ts as follows.
Let  \ts $\Qf:=(\Vf, \Ef)$ \ts be the (infinite) acyclic digraph
with the set of vertices \.  $\Vf \ts := \Vfm\cup \Vf^1 \cup \ldots \cup \Vf^{k-1}$ \.
given by
\begin{align*}
	\Vf^m \, & := \ \big\{ \ts (\alpha,m,t) \ \mid \  \alpha \in X^* \text{ with }  |\alpha| \leq k-1-m, \ t \in [0,1]  \ts \big\} \ \  \text{ for } \ m \in \{1,2, \ldots,k-1\}, \\
	\Vfm  \ & := \ \big\{ \ts (\alpha,0,1) \ \mid \  \alpha \in X^* \text{ with }  |\alpha| \leq k-1  \ts \big\}.
\end{align*}
Here the restriction \ts $t=1$ \ts in \ts $\Vfm$ \ts is crucial for a technical reason that will be apparent later in the section.

Let \. $\Xf:=X\cup \{\spstar\}$ \ts be the set of letters~$X$
with one special element \ts $\spstar$ \ts added.  The reader should think of
element \ts $\spstar$ \ts as the empty letter.  Let \ts $\ar:= |\Xf|=(n+1)$ \ts
be the dimension of the atlas.  Then each vertex \ts $\vf\in \Vf^m$, \ts $m\geq 1$,  has
exactly \ts $(n+1)$ \ts outgoing edges which we label \ts $\bigl(\vf,\vf^{\<x\>}\bigr)\in \Ef$, where
\ts $x\in \Xf$ \ts and \ts $\vf^{\<x\>}\in \Vf^{m-1}$ \ts are defined as follows:
\[ \
\vf^{\<x\>} \ := \
\begin{cases}
	(\alpha x,m-1,1) & \text{ if } \ \, x \in  X,\\
	(\alpha,m-1,1) & \text{ if } \ \, x = \spstar\ts.
\end{cases}
\]
Let us emphasize that this is not a typo and for all \ts $\vf^{\<x\>}$ \ts
we indeed have the last parameter \ts $t=1$, see Figure~\ref{f:matroid-edge}.


\begin{figure}[hbt]
	\includegraphics[width=8.1cm]{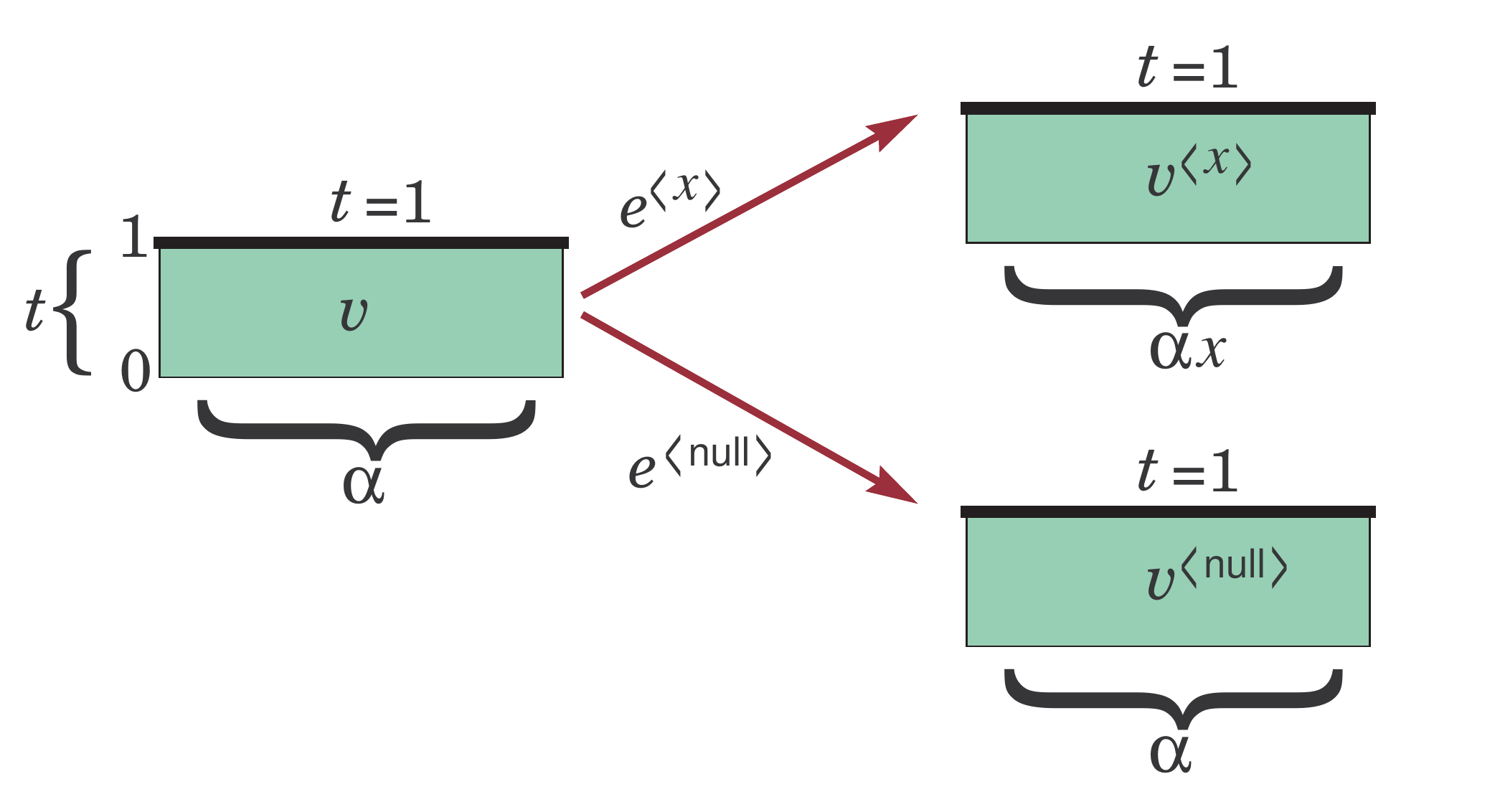}
	\vskip-.25cm
	\caption{{Atlas edges of two types: \. $\ef^{\<x\>} = \bigl(\vf,\vf^{\<x\>}\bigr)$, where \ts $v=(\al,m,t)$ \ts and \ts $\vf^{\<x\>}=(\al x,m-1,1)$, and \.
		$\ef^{\<\spstar\>} = \bigl(\vf,\vf^{\<\spstar\>}\bigr)$, where \ts $v=(\al,m,t)$ \ts and \ts $\vf^{\<\spstar\>}=(\al,m-1,1)$. }}
	\label{f:matroid-edge}
\end{figure}

\medskip

For every word \ts $\alpha \in X^*$ \ts of length \ts $\ell=|\al|$,
and for every \ts $1\le m \le \rk(\Mf)-\ell-1$,
denote by \.
$\bA(\alpha,m) := (\aA_{x\ts y})_{x,y \in \Xf}$ \. the symmetric \ts $\ar \times \ar$ \ts
matrix defined as follows:
\begin{equation}\label{eq:A-uni}
\ \aligned \ & \ \aA_{x\ts y} \, := \,  \bigl| \Cnt_{m-1}(\alpha xy)  \bigr|  \qquad  \text{for  \ $x,y \in \Cnt(\alpha)$,} 
\\
	\ & \ \aA_{x\.\spstar} \, = \, \aA_{\spstar \. x} \, := \,  \bigl| \Cnt_{m-1}(\alpha x)  \bigr|
	\qquad \text{for \ \ $x \in \Cnt(\alpha)$,} \\
	\ & \ \aA_{\spstar\. \spstar} \,  := \, \bigl| \Cnt_{m-1}(\alpha)  \bigr|\ts.
\endaligned
\end{equation}
For the first line, note that  \. $\aA_{x\ts y} = \aA_{y\ts x}$ \. by matroid symmetry property.
Note also that \. $\aA_{x\ts x} = 0$ \. for $x \in X$ since \ts $\alpha \beta x x $ \ts is not simple.
Next, observe that
\.  $\aA_{x\ts y} = 0$ \.
whenever  \. $x \notin \Cnt(\alpha) \cup \{\spstar\}$ \. or \.  $y \notin \Cnt(\alpha) \cup \{\spstar\}$.
Finally, we have \ts $\aA_{x\.\spstar} >0$ \ts whenever \ts $x \in \Cnt(\alpha)$, since by the exchange property
the word \ts $\alpha x \in \Lc$ \ts can be extended to \ts $\alpha x\beta \in \Lc$ \ts
for some \ts $\beta \in X^*$ \ts with \. $|\beta| \leq \rk(\Mf)-\ell-1$.

\smallskip

For each vertex \. $\vf=(\alpha,m,t) \in \Vf$, define the associated matrix as follows:
\[  \bM \, = \, \bM_{(\alpha,m,t)} \ := \  t \, \bA(\alpha,m+1)  \ +  \ (1-t) \, \bA(\alpha,m).
\]
Similarly, define the associated vector \. $\hb = \hb_{(\alpha,m,t)} \in \Rb^{\ar}$ \. with
coordinates
\[  \ah_x \ := \
\begin{cases}
	\, t & \text{ if } \ x \in X,\\
	\, 1-t & \text{ if } \ x =\spstar.
\end{cases}    \]
Finally, define
the linear transformation \. $\bT^{\<x\>}: \Rb^{\ar} \to \Rb^{\ar}$ \.
associated to the edge \ts $(\vf,\vf^{\<x\>})$ to be the identity map.

\medskip

\subsection{Properties of the atlas} \label{ss:greedoid-prop}
We now show that our combinatorial atlas \ts $\AA$ \ts
satisfies the conditions in Theorem~\ref{t:Hyp}, in the following series of lemmas.

\smallskip

\begin{lemma}\label{l:greedoid Irr}
	For every vertex \. $\vf\ts =\ts (\alpha,m,t)\in \Vf$, we have:
	\begin{enumerate}
		\item \label{item:Irr 1} the support \. $\supp(\bMr_{\vf} \, )$ \. of the associated matrix \ts $\bMr_{\vf}$ \ts is given by
		$$\supp\big(\bAr \, (\alpha,m+1)\big) \ = \ \supp\big(\bAr \, (\alpha,m)\big) \ = \
		\begin{cases}
			\Cnt(\alpha) \ts \cup \ts \{\spstarr\}  & \text{ if }\  \alpha \in \Lc\\
			\varnothing & \text{ if } \ \alpha \notin \Lc
		\end{cases}	 $$
		\item \label{item:Irr 2} vertex \. $\vf$ \. satisfies \eqref{eq:Irr}, and
		\item \label{item:Irr 3} vertex \. $\vf$ \. satisfies \eqref{eq:hPos} \. for all \. $t \in (0,1)$.
	\end{enumerate}
\end{lemma}

\smallskip

\begin{proof}
	Part \eqref{item:Irr 1} follows directly from the definition
	of matrices \. $\bM$, \. $\bA(\alpha,m+1)$ \. and \. $\bA(\alpha,m)$.
	For part \eqref{item:Irr 2}, observe that if \ts $\alpha \notin \Lc$, then $\bM$ is a zero matrix and
	\ts $\vf$ \ts trivially satisfies \eqref{eq:Irr}.
	On the other hand, if \ts $\alpha \in \Lc$,
	then it
	follows from the definition of \. $\bM = \bigl(\aM_{x\ts y}\bigr)$,
	that \ts $\aM_{x\. \spstar} >0$ \ts for every \ts $x \in \Cnt(\alpha)$.
	Since the support of~$\ts\bM$ \ts is \ts $\Cnt(\alpha) \ts \cup \ts \{\spstar\}$,
	this proves \eqref{eq:Irr}.
	Finally, part~\eqref{item:Irr 3} follows from the fact that \ts $\hb_{\vf}$ \ts is a strictly positive vector when \ts $t\in (0,1)$, and  that $\bM$ is a nonnegative matrix.
\end{proof}

%
%

\smallskip

\begin{lemma}\label{l:greedoid Inh}
	For every matroid \ts $\Mf=(X,\Ic)$, the atlas \ts $\AA$ \ts satisfies  \eqref{eq:Inh}.
\end{lemma}

\begin{proof}
	Let $\vf=(\alpha,m,t)\in \Vf^m$, $m\ge 1$, be a non-sink vertex of~$\Qf$.
	Let \. $x \in \Xf$.
	By the linearity of~$\bT^{\<x\>}$, it suffices to show that for every \.
	$y \in \Xf$, we have:
	\[  \aM_{xy} \ = \ \big\< \bT^{\<x\>} \eb_y\ts, \ts\bM^{\<x\>} \,  \bT^{\<x\>} \hb \big\>,
	\]	
	where \. $\big\{\eb_{y}, \. y \in \Xf\big\}$ \. is the standard basis in \ts $\Rb^{\ar}$.
	We present only the proof for the case \ts $x,y \in \Cnt(\alpha)$, as the proof
	of the other cases are analogous.	
	
	First suppose that \ts $x,y \in \Cnt(\alpha)$ \ts are distinct.
	Then:
	\begin{equation*}
		\big\< \bT^{\<x\>} \eb_y\ts, \ts\bM^{\<x\>} \,  \bT^{\<x\>} \hb \big\> \quad = \ \
		\sum_{z \ts\in\ts \Xf}  \aM^{\<x \>}_{yz} \, \big(\bT^{\<x\>} \hb \big)_z \  = \  \sum_{z \ts\in\ts X}
		\, t\. \bA(\alpha x, m)_{yz}  \  + \ (1-t)\. \bA(\alpha x, m)_{y\ts \spstar}.
	\end{equation*}
Let \ts $\ell:=|\al|$.  By the definition of \. $ \bA(\alpha x, m)$, this is equal to
	\begin{equation}\label{eq:sum-uni}
		\begin{split}			
			&  \  \sum_{z \ts\in\ts X}  \,    t \,  |\Cnt_{m-1}(\alpha xyz)|
			\ + \   (1-t) \,  |\Cnt_{m-1}(\alpha xy)|
			\\
			& \qquad \ = \   t \,  |\Cnt_{m}(\alpha xy)| \ +  \  (1-t) \,   |\Cnt_{m-1}(\alpha xy)|.
		\end{split}
	\end{equation}
	By the definition of \ts $ \bA(\alpha, m+1)$ \ts and \ts $ \bA(\alpha, m)$, this is equal to
	\begin{align*}
		t \, \bA(\alpha, m+1)_{xy} \ + \
		(1-t) \,  \bA(\alpha, m)_{xy} \ = \   \aM_{xy}\.,
	\end{align*}
	which proves \eqref{eq:Inh} for this case.
	This completes the proof.
\end{proof}

\smallskip

\begin{lemma}\label{l:greedoid TPInv}
	For every matroid \ts $\Mf=(X,\Ic)$, the atlas \ts $\AA$ \ts satisfies  \eqref{eq:TPInv}.
\end{lemma}

\smallskip

\begin{proof}
	Let $\vf=(\alpha,m,t)\in \Vf^m$, $m\ge 1$, be a non-sink vertex of~$\Qf$, and
	let \ts $x,y,z\in\Xf$.
	We present only the proof of \eqref{eq:TPInv} for the case when
	\. $x,y,z \in X$,  as other cases follow analogously.
	We have:
	\begin{align} \label{eq:TP-inv-uni}
		\aM_{yz}^{\<x\>} \ = \ \bA(\alpha x,m)_{yz} \ = \  |\Cnt_{m-1}(\alpha xyz)|\ts.
	\end{align}
	By the matroid symmetry property, the right side of the equation above is invariant under
	every permutation of $\{x,y,z\}$.
	This shows that \ts $\aM_{yz}^{\<x\>}$ \ts is invariant under every permutation of $\{x,y,z\}$, and the proof is complete.
\end{proof}

\smallskip

\begin{lemma}\label{l:greedoid Decsupp}
	For every matroid \ts $\Mf=(X,\Ic)$, the atlas \ts $\AA$ \ts satisfies  \eqref{eq:Decsupp}.
\end{lemma}

\smallskip

\begin{proof}
	Let $\vf=(\alpha,m,t)\in \Vf^m$, $m\ge 1$, be a non-sink vertex of~$\Qf$,
	and let $x \in \Xf$.
	We need to show that \. $\supp(\bM)  \supseteq  \supp(\bM^{\<x\>})$.
	First suppose that $\alpha \notin \Lc$.
	Then by Lemma~\ref{l:greedoid Irr}\eqref{item:Irr 1}
	\ts $\supp(\bM)  = \supp\big(\bM^{\<x\>}\big)= \varnothing$, and the lemma hold trivially.

	Now suppose that $\alpha \in \Lc$.
	By Lemma~\ref{l:greedoid Irr}\eqref{item:Irr 1},
	it suffices to show that \ts $\Cnt(\alpha) \supseteq \Cnt(\alpha x)$.
	Recall that for every \. $\alpha x y \in \Lc$ \. we have  \. $\alpha  y \in \Lc$ \.
by the matroid symmetry and hereditary properties.  This implies the result.
\end{proof}

\medskip

\subsection{Proof of Theorem~\ref{t:matroids-BH}} \label{ss:-proof}

We first show that every sink vertex in \ts $\Qf$ \ts is hyperbolic.

\smallskip

\begin{lemma}\label{l:greedoid BHyp}
	Let \ts $\Mf=(X,\Ic)$ \ts be a matroid on \ts $|X|=n$ elements,
	and let \ts $1\le k <\rk(\Mf)$.
	Then every vertex in \ts $\Vfm$ \ts satisfies \eqref{eq:Hyp}.
\end{lemma}

\begin{proof}
	Let \ts $\vf=(\alpha, 0,1)\in \Vfm$ \ts be a sink vertex, and let \ts $\ell:=|\al|$.
	It suffices to show that \ts $\bA(\alpha,1)$ \ts
	satisfies \eqref{eq:Hyp}. First note that if \ts
	$\alpha \notin \Lc$, then \ts $\bA(\alpha,1)$ \ts
	is a zero matrix, and \eqref{eq:Hyp} is trivially true.
	Thus, we can assume that \ts $\alpha \in \Lc$.
	We write \. $\aA_{x,y} \ts := \ts \bA(\alpha,1)_{xy}$ \.
	for every \ts $x,y \in X$.
	
	We define an equivalence relation on $\Cnt(\alpha)$ by writing \ts $x \sim y$ \ts if \ts $\alpha xy \not \in \Lc$.
	Note that reflexivity of the relation follows from the fact that $\alpha xx$ is not a simple word,
	symmetry follows from the matroid symmetry property, and
	transitivity follows from the exchange property.
	Note that the number of equivalence classes \ts $r$ \ts of this relation is at most
	\begin{equation}\label{eq:parallel class}
		r \ \leq \  |\Cnt(\alpha)| \  \leq \ n-\ell \ \leq \ n-k+1.
	\end{equation}

	Now, for every \ts $x\in \supp(\bA) \setminus \{\spstar \}$ \ts and $y\in \supp(\bA)$, we have:
\begin{equation}\label{eq:Hyp-uni}
\aA_{xy} \ = \
	\begin{cases}
		1 & \text{ if  \. $y \in X$ \. and \. $y \not \sim x$},\\
		0 & \text{ if  \. $y \in X$ \. and \. $y \sim x$},\\
		1 & \text{ if  \. $y =\spstar$}.
	\end{cases}
\end{equation}
In particular, this shows that the $x$-row (respectively, $x$-column)
of \ts $\bA(\alpha,1)$ \ts is identical to the $y$-row (respectively,
$y$-column) of \ts $\bA(\alpha,1)$ \ts whenever \. $x \sim y$.
	In this case,
	deduct the $y$-row and $y$-column of \ts $\bA(\alpha,1)$ \ts
	by the $x$-row and $x$-column of \ts $\bA(\alpha,1)$.
	It then follows from the claim that the resulting matrix has $y$-row and $y$-column is equal to zero.
Note that \eqref{eq:Hyp} is preserved under this transformation.

Now, apply the above linear transformation repeatedly,
and by restricting to the support of resulting matrix.  Since this
preserves \eqref{eq:Hyp},
it  suffices to prove that the following \ts $(r+1) \times (r+1)$ \ts  matrix   satisfies \eqref{eq:Hyp}:
	{\small
		\begin{align}\label{eq:B-uni}
			\bB \ = \ \begin{pmatrix}
				0 &   1 & \hspace{3 pt} \cdots \hspace{3 pt}
				& 1  & 1\\[2 pt] 		
				1 &  0
				& \cdots &   1 & 1  \\[2 pt]
				\vdots & \vdots & \ddots &  \vdots & \vdots \\[3 pt]
				1 & 1 & \cdots & 0  & 1 \\[3 pt]
				1 & 1 & \cdots & 1  & 1
			\end{pmatrix}\ts.
		\end{align}
	}

\nin
Note that \ts $\bB$ \ts has eigenvalue
\. $\la=-1$ \. with multiplicity~$(r-1)$.  Since \. $\det(\bB)=(-1)^r$,
we conclude that~$\bB$ has exactly one positive eigenvalue.
By Lemma~\ref{l:Hyp is OPE}, this implies
the result.
\end{proof}

\smallskip

We can now prove that every vertex in \ts $\Qf$ \ts is hyperbolic.

\smallskip

\begin{lemma}\label{t:abs}
	Let \ts $\Mf=(X,\Ic)$ \ts be a matroid on \ts $|X|=n$ \ts elements,
	and let \ts $1\le k <\rk(\Mf)$.
	Then every vertex in \ts $\Vf$ \ts satisfies \eqref{eq:Hyp}.
\end{lemma}

\begin{proof}
	We induction on~$m$ to show that every vertex in \ts $\Vf^m$ \ts
	satisfies \eqref{eq:Hyp}, for all \ts $m \leq k-1$.
	The claim is true for \ts $m=0$ \ts by Lemma~\ref{l:greedoid BHyp}.
	Suppose that the claim is true for \ts $\Vf^{m-1}$.
	Now note that the atlas \ts $\AA$ \ts satisfies all the necessary properties:  \ts
	\eqref{eq:Inh} by Lemma~\ref{l:greedoid Inh}, \ts
	\eqref{eq:TPInv} by Lemma~\ref{l:greedoid TPInv}, \ts
	\eqref{eq:Decsupp} by Lemma \ref{l:greedoid Decsupp},
	and \ts \eqref{eq:Iden} by definition.
	It then follows from Theorem~\ref{t:Hyp} that  every regular vertex in \ts $\Vf^m$ \ts satisfies \eqref{eq:Hyp}.
	
	On the other hand, by Lemma~\ref{l:greedoid Irr}, the regular vertices of \ts $\Vf^m$ \ts
	are those of the form \ts $\vf=(\alpha,m,t)$ \ts with \ts $t \in (0,1)$.
	Since \eqref{eq:Hyp} is preserved under taking the limits \ts $t\to 0$ \ts and \ts $t\to 1$,
	it then follows that every vertex in $\Vf^m$ satisfies \eqref{eq:Hyp}. This completes the proof.
\end{proof}

\smallskip

\begin{proof}[Proof of Theorem~\ref{t:matroids-HSW}]
	Let \ts $\bM=\bM_{\vf}$ \ts be the matrix associated with the vertex \ts $\vf=(\varnothing,k-1,1)$.	
	Let \ts $\vb$ \ts and~$\ts \wb$ \ts be the characteristic vectors of \ts $X$ \ts and \ts $\{\spstar\}$, respectively.
	Then:
\begin{equation}\label{eq:greedoid-last}
{\aligned
		\< \vb, \bM \vb \> \ = \  (k+1)! \,\,  \aI(k+1)\ts, \hskip1.8cm \\
\< \vb, \bM \wb \> \ = \ k! \,\.  \aI(k) \qquad \text{and} \qquad \< \wb, \bM \wb \> \ = \ (k-1)! \,\.  \aI(k-1).
\endaligned}
\end{equation}	
	By Lemma~\ref{t:abs}, vertex \ts $\vf$ \ts satisfies \eqref{eq:Hyp}.
	Substituting~\eqref{eq:greedoid-last} into \eqref{eq:Hyp}, gives the
	inequality  in the theorem.  \end{proof}
\medskip

\subsection{Ultra-log-concavity}
In this section we extend the proof above to the obtain the strong Mason conjecture~\cite{Mas},
which was recently established independently by Anari~et.~al~\cite{ALOV} and Br\"and\'en--Huh~\cite{BH}.

\smallskip

\begin{thm}[{\rm  \defng{Ultra-log-concavity for matroids}, \cite[Thm~1.2]{ALOV} and \cite[Thm~4.14]{BH}, formerly \defna{strong Mason conjecture}}{}]
	\label{t:matroids-BH}
	For a matroid \ts $\Mf=(X,\Ic)$,  \ts $|X|=n$,  \ts and integer \ts $1\le k < \rk(\Mf)$, we have:
	\begin{equation}\label{eq:matroid-ULC}
		\rI(k)^2 \, \ge \, \left(1 \. + \, \frac{1}{k}\right)  \left(1 \. + \, \frac{1}{n-k}\right)  \,\.  \rI(k-1)  \   \rI(k+1)\ts.
	\end{equation}
\end{thm}

\smallskip

\nin
Note that in~\cite{CP}, we used the same this proof technique to obtain an even stronger version
of \eqref{eq:matroid-ULC}, see \cite[\S1.4]{CP} for details. In this paper we present
only the proof of \eqref{eq:matroid-ULC} for simplicity.

\smallskip

\begin{proof}[Proof of Theorem~\ref{t:matroids-BH}]
We proceed verbatim the proof above with minor changes.  First, we modify the
definition~\eqref{eq:A-uni} of the symmetric \ts $\ar \times \ar$ \ts
matrix \ts $\bA(\alpha,m)$ \ts as follows:
\begin{align*}
	& \aA_{x\ts y} \, := \, \ac_{\ell+m+1} \,  \bigl| \Cnt_{m-1}(\alpha xy)  \bigr|  \qquad  \text{for  \ $x,y \in \Cnt(\alpha)$,}\\
	& \aA_{x\.\spstar} \, = \, \aA_{\spstar \. x} \, := \,  \ac_{\ell+m} \, \bigl| \Cnt_{m-1}(\alpha x)  \bigr|
	\qquad \text{for \ \ $x \in \Cnt(\alpha)$,}\\
	& \aA_{\spstar\. \spstar} \, := \, \ac_{\ell+m-1} \,  \bigl| \Cnt_{m-1}(\alpha)  \bigr|,
\end{align*}
where \. $\ac_{i} := 1+\frac{1}{n-k}$ \. if  \ts $i=k+1$,
and  \. $\ac_{i} := 1$ \. otherwise.  Then the intermediate equation~\eqref{eq:sum-uni} becomes
	\begin{equation*}
		\begin{split}			
			&  \  \sum_{z \ts\in\ts X}  \,    t \, \ac_{\ell+m+2} \,  |\Cnt_{m-1}(\alpha xyz)|
			\ + \   (1-t) \, \ac_{\ell+m+1} \,  |\Cnt_{m-1}(\alpha xy)|
			\\
			& \quad  = \   t \, \ac_{\ell+m+2} \,  |\Cnt_{m}(\alpha xy)| \ +  \  (1-t) \, \ac_{\ell+m+1} \,  |\Cnt_{m-1}(\alpha xy)|.
		\end{split}
	\end{equation*}
but the conclusion of Lemma~\ref{l:greedoid Inh} remains valid.

Similarly, the equation~\eqref{eq:TP-inv-uni} becomes
\begin{align*}
\aM_{yz}^{\<x\>} \ = \ \bA(\alpha x,m)_{yz} \ = \ \ac_{|\alpha|+m+2} \,  |\Cnt_{m-1}(\alpha xyz)|\ts.
\end{align*}
and the conclusion of  Lemma~\ref{l:greedoid TPInv} remains valid.

Finally, the proof of Lemma~\ref{l:greedoid BHyp} is more technical in this case.
First, the equation~\eqref{eq:Hyp-uni} becomes:
\begin{equation}\label{eq:Hyp-non}
\aA_{xy} \ = \
	\begin{cases}
		\ac_{\ell+2} & \text{ if  \. $y \in X$ \. and \. $y \not \sim x$},\\
		0 & \text{ if  \. $y \in X$ \. and \. $y \sim x$},\\
		\ac_{\ell+1} & \text{ if  \. $y =\spstar$}.
	\end{cases}
\end{equation}
Next, matrix \ts $\bB$ \ts in~\eqref{eq:B-uni} is now
{\small
		\begin{align}
			\bB \ = \ \begin{pmatrix}
				0 &   \ac_{\ell+2} & \hspace{3 pt} \cdots \hspace{3 pt}
				& \ac_{\ell+2} &    \ac_{\ell+1}  \\[2 pt] 		
				\ac_{\ell+2} &  0
				&  &   \vdots  &   \vdots  \\[2 pt]
				\vdots &  & \ddots &  \ac_{\ell+2} & \ac_{\ell+1}  \\[3 pt]
				\ac_{\ell+2} & \cdots & \ac_{\ell+2} &
				0 & \ac_{\ell+1} \\[3 pt]
				\ac_{\ell+1} & \ldots & \ac_{\ell+1} & \ac_{\ell+1} &  \ac_{\ell}
			\end{pmatrix}.
		\end{align}
	}

\nin
	Rescale the $i$-th row and $i$-th column ($i \leq r$)
	by \. $\frac{1}{\sqrt{\ac_{\ell+2}}}$,
	and the $(r+1)$-th row and $(r+1)$-th column by
	$\frac{\sqrt{\ac_{\ell+2}}}{\ac_{\ell+1}}$.
	Note that \eqref{eq:Hyp} is preserved under this transformation.
	The matrix becomes
{\small
	\begin{align*}
	\bB' \ = \	\begin{pmatrix}
			0 &   1 & \hspace{3 pt} \cdots \hspace{3 pt}
			& 1 &    1  \\[2 pt] 		
			1 &  0
			&  &   \vdots  &   \vdots  \\[2 pt]
			\vdots &  & \ddots &  1 & 1  \\[3 pt]
			1 & \cdots & 1 &
			0  & 1 \\[3 pt]
			1 & \ldots & 1 & 1 &  \frac{\ac_{\ell+2} \, \ac_{\ell}}{ \ac_{\ell+1}^2}
		\end{pmatrix}.
	\end{align*}
}	

\nin
By Lemma~\ref{l:Hyp is OPE}, it suffices to show that \ts $\bB'$ \ts
has exactly one positive eigenvalue.  Indeed, observe that \ts $\lambda=-1$ \ts
is an eigenvalue of this matrix with multiplicity \ts $(r-1)$.
	So this matrix has exactly one positive eigenvalue if and only if
	the determinant of the matrix has sign \ts $(-1)^r$.
	On the other hand, the determinant of this matrix is equal to
	\[ (-1)^r \, \bigg[ \frac{\ac_{\ell+2} \, \ac_{\ell}}{ \ac_{\ell+1}^2} (1-r) +r \bigg]. \]
	Now note that
$$
\frac{\ac_{\ell+2} \, \ac_{\ell}}{ \ac_{\ell+1}^2} \ = \ \left\{
\aligned 1 \qquad \ & \text{if \ \ $\ell <k-1$,} \\
1+\frac{1}{n-k} \quad \  & \text{if \ \ $\ell =k-1$.}
\endaligned\right.
$$
	In both cases, the determinant above has sign \ts $(-1)^r$ \ts since \ts $r \leq n-k+1$ \ts by \eqref{eq:parallel class}.
	This completes the proof of Lemma~\ref{l:greedoid BHyp} in this case ad finished the proof of property~\eqref{eq:Hyp}.

Finally, in the proof of Theorem~\ref{t:matroids-BH}, equation~\eqref{eq:greedoid-last} is modified as
\begin{equation}
{\aligned
		\< \vb, \bM \vb \> \ = \  \left(1+\frac{1}{n-k} \right) \, (k+1)! \,\,  \aI(k+1)\ts, \hskip1.8cm \\
\< \vb, \bM \wb \> \ = \ k! \,\.  \aI(k) \qquad \text{and} \qquad \< \wb, \bM \wb \> \ = \ (k-1)! \,\.  \aI(k-1).
\endaligned}
\end{equation}	
The rest of the proof is unchanged again.
\end{proof}

\medskip

\subsection{Abstract simplicial complex}
An \defna{abstract simplicial complex} is a pair $(X,\Delta)$, where $X$ is the \defng{ground set} and \. $\Delta \subseteq 2^X$  \. is a collection of subsets of $X$ that satisfies \defng{hereditary property}: \.
$S\subset T$, \. $T\in \Delta$ \, $\Rightarrow$ \, $S\in \Delta$\ts.
The subsets in \ts $\Delta$ \ts are called the \defnb{faces} of the simplicial complex.
Note that a matroid is an abstract simplicial complex that additionally satisfies
the \defng{exchange property}.  The \defnb{rank} of \ts $\Delta$,
denoted by \ts $\rk(\Delta)$, is the largest cardinality of any of its faces.
Note that the \defnb{dimension} $\dim(\Delta)$  of $\Delta$ is equal to $\rk(\Delta)-1$.

Let $k \in \{1,\ldots, \rk(\Delta)-1\}$.  We now define the combinatorial atlas
\ts $\AA(\Delta,k)$ \ts  verbatim the same way combinatorial atlas for a matroid
is defined in~$\S$\ref{ss:Combinatorial-atlas-matroid}.  Note that the exchange
property is never used in the definition, so when \ts $\Delta$ \ts is a matroid \ts $\Mf$,
the atlas \ts $\AA(\Delta,k)$ \ts is equal to the atlas \ts $\AA(\Mf,k)$.

\smallskip

\begin{thm}\label{t:simp}
Let \ts $\Delta$ \ts be a simplicial complex.  Suppose that
every vertex in \ts $\AA(\Delta,k)$ \ts satisfies \eqref{eq:Hyp}, for all \. $1\le k <\rk(\Delta)$.
Then \ts $\Delta$ \ts is a matroid.
\end{thm}

\smallskip

In other words, we show that the exchange property is necessary for the proof of
Lemma~\ref{t:abs}, so Theorem~\ref{t:simp} can be viewed as a converse to
Lemma~\ref{t:abs}.  The proof is based on he following quick calculation.

\smallskip

\begin{lemma}\label{l:abs}
In conditions of the theorem, let \ts $U \in \Delta$, and let \ts $x \in X \setminus U$ \ts such that
\ts $U\cup \{x\} \in \Delta$.  Then, for every distinct \ts $y,z \in X \setminus U$ \ts
such that \ts $U \cup \{y,z\} \in \Delta$, we have either \ts $U \cup \{x,y\} \in \Delta$ \ts
or \ts $U\cup \{x,z\} \in \Delta$.
\end{lemma}

\smallskip	
	
	\begin{proof}
		The claim is trivial when \ts $x=y$ \ts or \ts $x=z$,
		so we can assume that \ts $x,y,z\in X \setminus U$ \ts are all distinct.
		Let \ts $\alpha=x_1\ldots x_\ell \in \Lc$ \ts be any feasible word such that \ts $U=\{x_1,\ldots, x_\ell\}$, and let $\bA:=\bA(\alpha,1)$.
		It follows from the assumption of the theorem that $\bA$ satisfies \eqref{eq:Hyp}.
		Let $\bA'$ be the restriction of $\bA$ to the rows and columns indexed by \ts $\{x,y,z,\spstar\}$.
		Then \ts $\bA'$ \ts also satisfies \eqref{eq:Hyp}.
			Now, suppose to the contrary that \ts $U \cup \{x,y\}\notin \Delta$ \ts and
		\ts $U \cup \{x,z\}\notin \Delta$.
		Then
		\[
		\bA'  \ = \  \begin{bmatrix}
			0 & 0 & 0 & 1\\
			0 & 0 & 1 & 1\\
			0 & 1 & 0 & 1\\
			1& 1 & 1 & 1
		\end{bmatrix} \qquad \text{ or } \qquad
	\bA'  \ = \  \begin{bmatrix}
		0 & 0 & 0 & 1\\
		0 & 0 & \frac{1}{n-1} & 1\\
		0 & \frac{1}{n-1} & 0 & 1\\
		1& 1 & 1 & 1
	\end{bmatrix},
		\]
where $n:=|X|>1$. In both cases, matrix \ts $\bA'$ \ts has more than one positive eigenvalue.
Now Lemma~\ref{l:Hyp is OPE} gives a contradiction.
	\end{proof}
	
	\smallskip
	
\begin{proof}[Proof of Theorem~\ref{t:simp}]
Recall the statement of the exchange property:  \ts for all
\ts $S,T \in \Delta$ \ts such that \ts $|S| < |T|$,
there exists \ts $y \in T \setminus S$  \ts such that
\ts $S \cup \{y\} \in \Delta$.
We this	by induction on  \. $i:=|S \setminus T|$.
The base case \ts $i=0$ \ts is trivial, so suppose that \ts $i\ge 1$.

	Let \. $x \in S \setminus T$, and let \. $U:=S \setminus x$.
	Since \. $|U\setminus T|= i-1$, by the induction assumption
	there exists distinct \. $y,z \in T \setminus U$ \. such that \. $U \cup \{y,z\}\in \Delta$.
	By the lemma above, it then follows that either \. $U\cup \{x,y\} \in \Delta$ \.
 or \. $U\cup \{x,z\} \in \Delta$.  By relabeling if necessary, we can assume that
\. $U\cup \{x,y\} \in \Delta$.
Then we have \. $S \cup \{y\} = U \cup \{x,y\}$, which proves the exchange property,
as desired.
\end{proof}

\bigskip

\section{Lorentzian polynomials}\label{s:Lorentzian}

In this section we will show that the theory of Lorentzian polynomials introduced by
Br\"and\'en and  Huh in \cite{BH}, can be expressed as a special case of our theory
of the combinatorial atlas.\footnote{In Anari~et.~al~\cite{ALOV}, a related notion of
\emph{strongly log-concave polynomials} was introduced.  We will focus only on
Lorentzian polynomials for simplicity of exposition.}
We refer to the aforementioned paper for further references on this topic.

\subsection{Background}
Let \ts $n,\ir\ge 0$ \ts be nonnegative integers.
We denote \ts $H_n^\ir$ \ts the set of degree $\ir$ homogeneous polynomials in \ts $\Rb[w_1,\ldots, w_n]$.
The \defn{Hessian} of $f \in \Rb[w_1,\ldots, w_n]$ is the symmetric matrix
\[ \Hr_f(w) \ := \  \big(\partial_i \partial_j f\big)_{i,j=1}^n,\]
where $\partial_i$ stands for the partial derivative \ts $\frac{\partial}{\partial w_i}$.

For every $\mb=(m_1,\ldots, m_n) \in \nn^n$, we write
\[ w^{\mb} \ := \ w_1^{m_1} \ldots w_n^{m_n} \quad \text{ and } \quad \partial^{\mb} \ := \ \partial_1^{\mb_1} \ldots  \partial_n^{\mb_n}.  \]
The \defn{$\ir$-th discrete simplex} $\Delta_n^{\ir} \subseteq \nn^n$ is
\[
\Delta_n^{\ir} \ := \  \bigl\{ \mb \in \nn^n \, : \, m_1 + \ldots + m_n = \ir    \bigr\}.
\]
The \defn{support} of a polynomial $f$ is the subset of $\nn^n$ defined by
\[ \supp(f) \ := \  \{ \mb \in \nn \, : \, \text{the coefficient of $w^{\mb}$ in $f$ is nonzero} \}. \]

A subset $J \subseteq \nn^n$ is \defn{M-convex} if,
for every $\mb,\nb \in J$ and every \ts $i\in [n]$ \ts s.t.\ $m_i > n_i$, there exists \ts $j\in [n]$ \ts
s.t.\  \. $m_j < n_j$ \. and \. $\mb-\eb_i+\eb_j \in J$, where $\eb_1,\ldots, \eb_n$ is the standard basis in~$\Rb^n$.

\smallskip

\begin{definition}\label{d:Lorentzian}
	Let $f$ be a homogeneous polynomial of degree $\ir$ with nonnegative coefficients.
	Then $f$ is \defn{Lorentzian} if the support of $f$ is M-convex, and
		the Hessian of \ts $\partial^{\mb} f$ \ts has at most one positive eigenvalue,
for every \ts $\mb \in \Delta_n^{\ir-2}$.
\end{definition}

\medskip

\subsection{Combinatorial atlas for Lorentzian polynomials}

In this section, we  define a combinatorial atlas that arises naturally from Lorentzian polynomials.
As a byproduct of this identification, we recover the following basic fact of Lorentzian polynomials.

\smallskip

\begin{thm}[{\rm cf.\ {\cite[Theorem~2.16(2)]{BH}}}{}]\label{thm:Lorentzian}
	Let $f$ be a Lorentzian polynomial. Then the Hessian of $f$ satisfies \eqref{eq:Hyp} for every $(w_1,\ldots, w_n)\in \Rb^n_{>0}$.
\end{thm}

\smallskip

Let $f \in H_n^\ir$ be a Lorentzian polynomial with $\ir \geq 3$,
and let $(w_1,\ldots, w_n) \in \Rb^n_{>0}$.
We define a combinatorial atlas \ts $\AA:=\AA(f, w_1,\ldots, w_n)$ \ts  as follows.
Let \ts $\ar:=n$ \ts be the dimension of the atlas, and let
\ts $X=[n]$.
Let  \ts $\Qf:=(\Vf, \Ef)$ \ts be the acyclic digraph where the set of vertices \.  $\Vf \ts := \Vfm\cup \Vf^1 \cup \ldots \cup \Vf^{\ir-2}$ \.
is given by
\begin{align*}
	\Vf^m \, & := \ \big\{ \ts \alpha \in X^* \ :   \ |\alpha| \ts = \ts \ir-2 -m \ts \big\},
\quad \text{for} \ \   0 \leq m \leq \ir-2\ts.
\end{align*}
Each  vertex \ts $\vf=\alpha\in \Vf^m$, \ts $m\geq 1$,  has
exactly  $n$ outgoing edges we label \ts $\bigl(\vf,\vf^{\<x\>}\bigr)\in \Ef$, where \ts $\vf^{\<x\>} = \alpha x$ \ts for every
\ts $x\in X$.

For each vertex \ts $\vf=\alpha$, where \ts $\al =x_1\ts\cdots\ts x_{\ir-2-m} \in \Vf^{m}$, \ts $m\geq 0$,  define the associated matrix \ts
$\bM_v$ \ts as follows:
\[ \bM_v \ := \ \Hr_f( \partial^{\alpha}w),  \qquad \text{ where } \qquad \partial^{\alpha} \ := \ \partial_{x_1} \ldots \partial_{x_{\ir-2-m}}.  \]
Define the associated vector \ts $\hb_v$ \ts for  $m\geq 1$, as follows:
\[  \hb_v \ := \  \, \Bigl(\frac{w_1}{m}\.,\.\ldots \., \. \frac{w_n}{m} \Bigr).  \]
Finally, define
the linear transformation \. $\bT^{\<x\>}: \Rb^{n} \to \Rb^{n}$ \.
associated to the edge \ts $\bigl(\vf,\vf^{\<x\>}\bigr)$, to be the identity map.

\medskip

\subsection{Proof of Theorem~\ref{thm:Lorentzian} }

We first verify that $\AA$ satisfies all conditions in Theorem~\ref{t:Hyp}.

\smallskip

\smallskip

\begin{lemma}\label{l:Lor-Irr}
	For every vertex \. $\vf\ts =\ts \alpha \in \Vf$, we have:
	\begin{enumerate}
		\item \label{item:Lor-Irr 1} the support of the associated matrix \ts $\bMr_{\vf}$ \ts  is given by
$$
\supp(\bMr_{\vf} \, ) \ = \ \bigl\{\ts i \in [n] \, : \,  \partial_i \partial^{\alpha} f  \neq 0 \ts\bigr\},
$$
		\item \label{item:Lor-Irr 2} vertex \. $\vf$ \. satisfies \eqref{eq:Irr}, and
		\item \label{item:Lor-Irr 3} vertex \. $\vf\in \Vf^m$ \. satisfies \eqref{eq:hPos}, for \ts $m\geq 1$.
	\end{enumerate}
\end{lemma}

\smallskip

\begin{proof}
	For
	part \eqref{item:Lor-Irr 1},
	note that $i \in [n]$ is not contained in \ts $\supp(\bM_v)$ \ts if and only if
	\. $ \partial_i  \partial_j \partial^{\alpha } f =0$ \. for every \. $j \in [n]$ by definition.
	Since $f$ is a homogeneous polynomial with nonnegative coefficients,
	the latter is equivalent to $ \partial_i  \partial^{\alpha} f =0$,
	and the claim follows.

For part \eqref{item:Lor-Irr 2}, denote by \. ``$\sim$'' \. the equivalence relation on the support of \ts $\bM=\bM_v$,
where two elements are equal if they are contained in the same irreducible component of $\bM$.
	Let us show that \ts $i \sim j$, for all \. $i,j \in  \supp(\bM)$.
	By part~\eqref{item:Lor-Irr 1},  there exist \. $\mb, \nb \in \nn^n$ \. in the support of $\partial^\al f$,
such that \ts $\am_i>0$ \ts and \ts $\an_j >0$.

\medskip

\nin
\textbf{Claim:} \. {\em
Every element $k$ in the support of \ts {\em $\mb$} \ts satisfies \. $k \sim i$.}

\begin{proof}[Proof of Claim]
The claim is clear for \ts $k=i$, so suppose that \ts $k \neq i$.
Then \. $\partial_{i}\partial_k w^{\mb} >0$. Since $\mb$ is contained in the support of \ts $\partial^\al f$,
		this implies that $\aM_{ik}>0$, and the claim now follows.
	\end{proof}
	
	If \ts $\mb=\nb$, then
	$i \sim j$  by the claim, so suppose to the contrary that \ts $\mb\ne \nb$.
	Then there exists $k \in [n]$ such that \ts $\am_k > \an_k \geq 0$.
	Since $f$ is M-convex, there exists $\ell \in [n]$, such that \ts $\an_\ell > \am_\ell \geq 0$ \ts
	and \ts $\mb-\eb_k+\eb_\ell$ \ts  is contained in the support of~$f$.
	We now show that \ts $i \sim \ell$.
	Indeed,
	If \ts $k \neq i$, then  \. $i \in \supp(\mb-\eb_k+\eb_\ell)$, which by the claim implies that $i \sim \ell$.
	If \ts $k=i$, then
	there exists \. $h \in \supp(\mb-\eb_k)$ \. since
	the  $\deg(w^{\mb})=\deg(\partial^{\alpha}f)$ is at least~2.
	Then \ts $i \sim h$ \ts and \ts $h \sim \ell$, by applying the claim to \ts $\mb$ \ts
	and \ts $\mb-\eb_k+\eb_\ell$, respectively.
	By transitivity, we then have \ts $i \sim \ell$, as desired.
	
	On the other hand, we have \ts $\ell \sim j$ \ts by the claim applied to~$\nb$.
	By transitivity, it then follows that \ts $i \sim j$,
	as desired.  Since $i$ and $j$ are arbitrarily chosen, this shows that~$\bM$
    is irreducible when restricted to its support, as desired.
	
	Finally, part \eqref{item:Lor-Irr 3} follows directly from the fact that $\hb_v$
is strictly positive by definition, and the fact that $\bM$ is nonnegative.
\end{proof}

\smallskip

\begin{lemma}\label{l:Lorentzian-conditions}
	For every Lorentzian polynomial \ts $f$, the atlas \ts $\AA$ \ts satisfies  \eqref{eq:Inh}, \eqref{eq:TPInv}, \eqref{eq:Iden}, and \eqref{eq:Decsupp}.
\end{lemma}

\smallskip

\begin{proof}
	Let $\vf=\alpha\in \Vf^m$, $m\ge 1$, be a non-sink vertex of~$\Qf$.
	Note that \eqref{eq:Iden} follows directly from definition.
	For~\eqref{eq:TPInv}, let \ts $i,j,k\in [n]$ \ts be arbitrary elements.
	Then:
\[
\aM^{(i)}_{jk} \ = \  \partial_i \partial_j \partial_k \partial^\al f, \quad \aM^{(j)}_{ki}
\ = \  \partial_j \partial_k \partial_i \partial^\al f, \quad \aM^{(k)}_{ij}
\ = \  \partial_k \partial_i \partial_j \partial^\al f,
\]
	which are all equal since partial derivatives commute with each other.
	This proves \eqref{eq:TPInv}.
	
	For~\eqref{eq:Decsupp}, let \ts $i \in [n]$.
	By Lemma~\ref{l:Lor-Irr}\eqref{item:Lor-Irr 1},
	the condition states that
 \. $ \partial_j \partial^{\alpha}  f =0$ \. implies \. $\partial_j \partial_i \partial^{\alpha}  f =0$,
 for every $j \in [n]$.  This is clear by commutativity again.

	It remains to verify~\eqref{eq:Inh}.
	First note that for every homogeneous polynomial \ts $g$ \ts of degree \ts $m \geq 1$, we have:
	\begin{equation}\label{eqmike 1}
		g \ = \  \frac{1}{m}\. \sum_{i=1}^n \. w_i \ts \partial_i g.
	\end{equation}
	Let \ts $i \in [n]$ \ts and \ts $\vb \in \Rb^n$.
	We have:
	\begin{align*}
		&  \big \langle   \bT^{\<i\>}\vb, \, \bM^{\<i\>} \,
		\bT^{\<i\>}\hb \big \rangle    \  =_{\eqref{eq:Iden}}
		\big \langle   \vb, \, \bM^{\<i\>} \,
		\hb \big \rangle \ = \
		\sum_{j=1}^n \. \sum_{k=1}^n \. \av_j \. \frac{w_k}{m} \, \partial_j \partial_k \partial_i \partial^{\alpha} f
		\\ & \qquad = \  \sum_{j=1}^n \. \av_j \. \frac{1}{m}\. \sum_{k=1}^n  \. w_k \, \partial_k \,  \Bigl( \partial_i \partial_j \partial^{\alpha} f \Bigr)
		\ =_{\eqref{eqmike 1}} \  \sum_{j=1}^k \. \av_j \, \partial_j \partial_i \partial^\al f \ = \  \big( \bM \vb\big)_{i}\..
	\end{align*}
	This proves \eqref{eq:Inh}.
\end{proof}

\smallskip

\smallskip

\begin{proof}[Proof of Theorem~\ref{thm:Lorentzian}]
	Note that the atlas $\AA$ satisfies every condition of Theorem~\ref{t:Hyp}
	by Lemma~\ref{l:Lorentzian-conditions}.
	Note also that every non-sink vertex of $\AA$ is regular by Lemma~\ref{l:Lor-Irr}.
	Applying Theorem~\ref{t:Hyp} iteratively,
	it suffices to show that the Hessian of \ts $\partial^{\alpha} f$ \ts satisfies \eqref{eq:Hyp}
	for every \ts $|\alpha|=\ir -2$.
	However, this is the assumption of Lorentzian polynomial, and the theorem now follows.
\end{proof}

\bigskip

\section{Alexandrov--Fenchel inequality}
\label{s:AF ineq}

In this section we give an elementary self-contained proof of the classical Alexandrov--Fenchel
inequality.  As we mentioned in the introduction, despite the difference in presentation,
the heart of the argument follows the proof in \cite[Thm~5.2]{SvH19} combined with a few
geometric arguments based on the presentation in~\cite{Sch}.

\medskip

\subsection{Mixed volumes}\label{ss:mixed volume}
Fix \ts $n \geq 1$.
For two sets \. $A, B \subset \Rb^n$ \. and constants \ts $a,b>0$, denote by
\[ aA+bB \ := \ \bigl\{ \ts a\xb+ b\yb  \, : \, \xb \in A, \yb \in B  \ts \bigr\}
\]
the \defnb{Minkowski sum} of these sets.
For a  convex body \ts $K \subset \Rb^n$, denote by \ts $\Vol_n(K)$ \ts the
volume of~$K$.
One of the basic result in convex geometry is \defng{Minkowski's theorem}
that the volume of convex bodies in~$\Rb^n$ behaves as a homogeneous polynomial
of degree~$n$ with nonnegative coefficients:

\medskip

\begin{thm}[{\rm Minkowski, see e.g.~\cite[$\S$19.1]{BZ-book}}{}]
For all convex bodies \. $\iK_1, \ldots, \iK_r \subset \Rb^n$ \. and \. $\lambda_1,\ldots, \lambda_r > 0$,
we have:
\begin{equation}\label{eq:mixed volume definition}
	\Vol_n(\lambda_1 \iK_1+ \ldots + \lambda_r \iK_r) \ =  \ \sum_{1 \ts \le \ts i_1\ts ,\ts \ldots \ts , \ts i_n\ts \le \ts r} \. \iV\bigl(\iK_{i_1},\ldots, \iK_{i_n}\bigr) \, \lambda_{i_1} \ts\cdots\ts \lambda_{i_n},
\end{equation}
where the functions \ts $\iV(\cdot)$ \ts are nonnegative and symmetric.
\end{thm}

\medskip

The coefficients \. $\iV(\iK_{i_1},\ldots, \iK_{i_n})$ \. are called \defnb{mixed volumes}
of \. $\iK_{i_1}, \ldots, \iK_{i_n}$.  They are \emph{invariant under translations}:
\[ \iV(\iK_{1}+\ba_1,\ldots, \iK_n + \ba_n )   \ = \ \iV(\iK_{1},\ldots, \iK_n) \quad \, \text{for every \ \. $\ba_1,\ldots, \ba_n \in \Rb^n$.}  \]
From this point on, every convex body in this section is assumed to be equivalent under translations.
Note also that \. $\iV(\iK,\ldots, \iK) = \Vol_n(\iK)$ \. for every convex body \. $K \subset \Rb^n$, and that \.
$\iV(\cdot)$ \. is \emph{multilinear}:
$$\iV(\lambda \iK+ \lambda' \iK',\iK_2,\ldots, \iK_n) \ = \  \lambda \. \iV(\iK,\iK_2,\ldots, \iK_n) \. + \. \lambda' \. \iV(\iK',\iK_2,\ldots, \iK_n) \quad \text{for every \ $\lambda, \lambda' > 0$.}
$$
Finally, mixed volumes are \emph{monotone}:
$$
\iV(\iK_1,K_2,\ldots, \iK_n) \, \ge \, \iV(\iK_1',\iK'_2,\ldots, \iK'_n) \quad \text{for all} \ \ \iK_i\supseteq \iK_i'\., \ 1\le i \le n.
$$

\smallskip

We will not prove Minkowski's theorem which is elementary and well presented in a number of
textbooks, such as \cite[$\S$8.3]{Ale}, \cite[$\S$19.1]{BZ-book}, \cite[$\S$5.1]{Sch}, and most recently
in \cite[$\S$3.3]{HW}. Instead, we will be concerned with the following classical inequality:

\medskip

\begin{thm}[{\rm \defng{Alexandrov--Fenchel inequality}, see e.g.~\cite[$\S$20]{BZ}\ts}{}]\label{thm:AF}
	For  convex bodies \ts $\iA$, $\iB$, $\iK_1$, \ts $\ldots$ \ts , $\iK_{n-2} \subset \Rb^n$, we have:
\begin{equation}\label{eq:AF} \tag{AF}
\iV(\iA,\iB, \iK_1, \ldots, \iK_{n-2})^2  \ \geq \ \iV(\iA,\iA, \iK_1, \ldots, \iK_{n-2}) \.\cdot\. \iV(\iB,\iB, \iK_1, \ldots, \iK_{n-2}).
\end{equation}
\end{thm}

\medskip

The way our proof works is by establishing hyperbolicity of a certain matrix
(Theorem~\ref{t:AF-Hyp}), which is where our combinatorial atlas technology comes in.
Unfortunately, both the matrix and the proof emerge in the middle of a technical calculations some
of which are standard, which go back to Minkowski and Alexandrov, and are widely available in the literature.
In an effort to make the proof self-contained, we made a choice to include them all,
sticking as much as possible to the presentation in Chapters~2 and~5 of \cite{Sch}.

\medskip

\subsection{Mixed volume preliminaries}
In this section we collect basic properties of mixed volumes that will be used in
the proof of Theorem~\ref{thm:AF}.  The reader well versed with  mixed volumes can
skip this subsection.

\smallskip

Let \ts $\iW \subseteq \Rb^n$ \ts be a linear subspace of dimension \ts $\dim(\iW)=m\geq 1$.
All polytopes in this paper are assumed to be convex.  A convex polytope
\ts $P\ssu W$ \ts is called \defn{$m$-dimensional} if it has nonempty interior.
An $m$-dimensional polytope $\iP \subset \iW$ is \defnb{simple}
if each vertex is contained in exactly $m$ facets.
The following easy lemma proves positivity of mixed volumes of simple polytopes.

\smallskip

\begin{lemma}[{\rm cf.~\cite[Thm~5.1.8]{Sch}}{}]\label{lem:mixed volume positivity}
	Let \ts $\iP_1,\ldots, \iP_m \subset \iW$ \ts be convex $m$-dimensional polytopes.
	Then
	\. $\iV(\iP_1,\ldots, \iP_m)>0$.
\end{lemma}

\smallskip

\begin{proof}
	Since $\iP_1,\ldots, \iP_{m}$ are $m$-dimensional,
	there exists line segments \. $\iS_i \subset \iP_i$\ts, \ts $1\le i\le m$,
	such that \. $\iS_1,\ldots, \iS_m$ \. span $\iW$.
	By the monotonicity of mixed volumes, we have:
	\[ \iV(\iP_1,\ldots, \iP_m) \ \geq \ \iV(\iS_1,\ldots, \iS_m).  \]
	On the other hand, by direct calculation, we have:
	\[ \iV(\iS_1,\ldots, \iS_m) \ = \ \frac{1}{m!} \, \Vol_{m}(\iS_1+\ldots+\iS_m)  \ > \ 0. \]
	since \. $\iS_1+\ldots+\iS_m$ \. is an $m$-dimensional parallelepiped
with positive volume.
\end{proof}

\smallskip

Denote by \ts $\iF(\iP,\ub)$ \ts the face of the polytope $\iP$ with normal direction \ts $\ub \in \Sb^{n-1}$.
We say that polytopes \. $\iP, \iP' \subset \iW$  \. are \defnb{strongly isomorphic} if
\[ \dim \iF(\iP,\ub) \ = \   \dim \iF(\iP',\ub) \quad \text{ for all } \ub \in \Sb^{n-1} \ts \cap \ts \iW.
\]
This is a very strong condition which implies that polytopes $\iP$ and~$\iP'$ are combinatorially
equivalent (have isomorphic face lattices), with the corresponding faces parallel to each other.
Being strongly isomorphic is an equivalence relation  on polytopes in $\iW$,
and the equivalence classes of this relation are called \defn{$a$-types}.

For the rest of this section, let \ts $\Ac$ \ts be a fixed $a$-type of $\iW$.
Let   \. $\ub^{\<1\>}, \ub^{\<2\>},\ldots$ \. be the  unit vectors in $\iW$ normals
to facets of polytopes in~$\Ac$, so we have:
\[ \dim \iF\big(\iP,\ub^{\<i\>}\big) \ = \  m-1 \quad \text{ for all } \ \iP \in \Ac.  \]
Denote by \ts $\iW^{\<i\>}$ \ts the hyperplane in $\iW$ that contains the origin \ts $\zero\in \iW$, and
is orthogonal to \ts $\ub^{\<i\>}.$

For a  polytope \ts $\iP \in \Ac$,
the \defnb{support vector} \. $\hb_{\iP} \in \Rb^{\ar}$ \. is  given by
\[
\big(\hb_{\iP} \big)_i \ := \ \sup_{\xb \in \iP} \ \<\ub^{\<i\>},\xb \>,
\]
the distance to the origin~$\zero$ of the supporting hyperplane of $\iP$ whose normal direction is $\ub^{\<i\>}$.
Note that the polytope \ts $\iP \in \Ac$ \ts is uniquely determined by the support vector \ts $\hb_{\iP}$.
Note also that \ts $\hb_{\iP}$ \ts is strictly positive if and only if~$\zero$ is contained in the interior of \ts $\iP$.
Finally, note that support vectors convert Minkowski sum into scalar sum, i.e.\
\. $\hb_{a\iA+b\iB} \. = \.  a \hb_{\iA} \ts  + \. b \hb_{\iB}$ \. for all \ts $a,b\ge 0$.

\smallskip

The next lemma shows that every vector in \ts $\Rb^{\ar}$ \ts
can be expressed as a linear combination of support vectors.
\smallskip

\begin{lemma}[{\rm cf.~\cite[Lem.~2.4.13]{Sch}}{}]\label{lem:h-vector span}
	Let  \ts $\iP\in \Ac$ \ts be a simple polytope. Then
	there exists \ts $\ep >0$, such that for every \ts $\vb \in \Rb^{\ar}$ \ts with $|\vb| <\ep$, we have:
	\[ \vb \ = \ \hb_{\iQ} -\hb_{\iP}\ts,  \]
	for some simple polytope \ts $\iQ \in \Ac$ \ts strongly isomorphic to~$\iP$.
\end{lemma}

\smallskip

\begin{proof}
	Let
	\[ \iQ \ := \  \bigl\{ \. \xb \in \iW  \, : \, \langle \ub^{\<i\>}, \xb\rangle \leq  \big(\hb_{\iP} + \vb\big)_{i}   \ \text{ for every } i \in [\ar] \. \bigr\}
\]
	be a polytope formed by translating the supporting hyperplanes of~$\iP$.
	Note that the property of being simple and being contained in $\Ac$ is preserved under small enough pertubations.
	The conclusion of the lemma now follows.
\end{proof}

\smallskip

For  every distinct  \ts $i,j \in [\ar]$\ts,
 denote by \. $\theta^{\<ij\>} \in [0,\pi]$  \. the angle between
$\ub^{\<i\>}$ and $\ub^{\<j\>}$, so we have
\. $\cos\theta^{\<ij\>} \. = \. \big\<\ub^{\<i\>}, \ub^{\<j\>}\big\> $.
Let \ts $\iP_1,\ldots, \iP_{r} \in \Ac$ \ts be simple, strongly isomorphic polytopes.
For every \ts $k \in [r]$, we write
\begin{equation}\label{eq: facet definition}
	\iF_k^{\<i\>} \ := \ \iF\big(\iP_k, \ub^{\<i\>}\big) \quad \text{and} \quad \iF_k^{\<ij\>} \ := \  \iF_k^{\<i\>} \cap \iF_k^{\<j\>},
\end{equation}
for the faces of $\iP_k$ that corresponds to the normal direction \ts $\ub^{\<i\>}$ \ts
and for the pair of directions \ts $\bigl\{\ub^{\<i\>}, \ub^{\<j\>}\bigr\}$, respectively.
By definition, we have \. $\theta^{\<ij\>}= \theta^{\<ji\>}$ \. and \. $\iF_k^{\<ij\>} \. = \. \iF_k^{\<ji\>}$.
When \ts $r=1$, i.e.\ there is only one polytope \ts $\iP=\iP_1$, we omit subscript~$1$ from the notation.
\smallskip

The next lemma show that the properties of being  simple and strongly isomorphic are inherited by the facets of the polytopes.

\smallskip

\begin{lemma}[{\rm cf.~\cite[Lemma~2.4.10]{Sch}}{}]\label{lem:ssi hereditary}
	Let  \. $\iP_1, \iP_2 \in \Ac$ \. be simple strongly isomorphic $m$-dimensional
polytopes.  Suppose \ts $\iF_1$ \ts and \ts $\iF_2$ \ts are facets
of \ts $\iP_1$ \ts and $\iP_2$ \ts  corresponding to  the same normal direction.
	Then \. $\iF_1$ \ts and \ts  $\iF_2$ \.
	are simple, strongly isomorphic $(m-1)$-dimensional polytopes.
\end{lemma}

\smallskip

\begin{proof}
The case $m=1$ is trivial, so we assume that $m \geq 2$.
By definition, the facets of a simple $m$-dimensional polytope are again
simple $(m-1)$-dimensional polytopes.
It thus suffices to show that \ts $\iF_1$ \ts and \ts $\iF_2$ \ts are strongly isomorphic polytopes.
	Let \ts $\iF_1'$ \ts and \ts $\iF_2'$ \ts be faces of \ts $\iF_1$ \ts and \ts $\iF_2$ \ts with the same unit normal direction.
	Then there exists \ts $\ub \in \Sb^{n-1}$,
	such that  \ts $\iF_1'$ \ts and \ts $\iF_2'$ \ts are faces of \ts $\iP_1$ \ts and \ts $\iP_2$ \ts with the normal direction $\ub$.
	Since \ts $\iP_1$ \ts and \ts $\iP_2$ \ts are strongly isomorphic, it then follows that
	\[ \dim \iF_1' \ = \  \dim \ts \iF(\iP_1,\ub) \ = \ \dim\ts\iF(\iP_2,\ub) \ = \  \dim \iF_2'\., \]
	as desired.
\end{proof}

\smallskip

Now let \ts $m\geq 2$, let \. $\iP\in \Ac$ \. be a simple polytope, and
and let  \. $\iJ =\iJ(\Ac)$  \. be  given by
\[ \iJ \ := \   \bigl\{ \ts (i,j)\in [\ar]^2  \ \ \ \text{s.t.} \ \ \dim \iF^{\<ij\>}  \ = \ m-2 \bigr\}.\]
Note that \ts $\iJ$ \ts does not depend on the choice of \. $\iP\in \Ac$ \.
since the facets of polytopes in \ts $\Ac$ \ts are strongly isomorphic by Lemma~\ref{lem:ssi hereditary}.
Also note that \. $\theta^{\<ij\>} \notin \{0,\pi\}$ \. for all \ts $(i,j) \in \iJ$,
so both \ts $\csc \theta^{\<ij\>}$ \ts and \ts $\cot \theta^{\<ij\>}$ \ts are well defined.

\smallskip

The next lemma relates the support vector of a polytope to the support vector of its facets.

\smallskip

\begin{lemma}[{\rm cf.~\cite[Eq.~(5.4)]{Sch}}{}]\label{lem:h-vector hereditary}
	Let $m \geq 2$, let \. $\iP \in \Ac$ \. be an $m$-dimensional  polytope.
	Then, for all $(i,j) \in \iJ$,
	\begin{equation*}
		\big(\hb_{\iF^{\<i\>}}\big)_j \ = \ \big(\hb_{\iP}\big)_j \csc \theta^{\<ij\>} \ - \
		\big(\hb_{\iP}\big)_i \cot \theta^{\<ij\>}.
	\end{equation*}
\end{lemma}

\smallskip

\begin{proof}
	Let \. $\ub^{\<ij\>} \perp \ub^{\<i\>}$ \. be the unit normal vector of the $(m-1)$-polytope \ts $\iF^{\<i\>}$ \ts at its $(m-2)$-face \. $\iF^{\<ij\>}$.
	Note that
	\begin{equation*}
		\ub^{\<j\>} \ = \  \ub^{\<i\>} \. \cos \theta^{\<ij\>} \ + \ \ub^{\<ij\>} \. \sin \theta^{\<ij\>}.
	\end{equation*}
	This implies that
	\begin{align}\label{eq:Fij}
		\big(\hb_{\iF^{\<i\>}}\big)_{j} \ = \   \sup_{\xb \in \iF^{\<i\>}}  \<  \ub^{\<ij\>}, \xb  \> \ = \  \sup_{\xb \in \iF^{\<i\>}}
		\<  \ub^{\<j\>}, \xb  \> \. \csc \theta^{\<ij\>} \ - \  \<  \ub^{\<i\>}, \xb  \> \. \cot \theta^{\<ij\>}.
	\end{align}
	On the other hand, we have
	\begin{align}\label{eq:Fij-sup}
		\sup_{\xb \in \iF^{\<i\>}}
		\<  \ub^{\<j\>}, \xb  \> \ = \  \big(\hb_{\iP}\big)_j \qquad \text{ and } \qquad
		\<  \ub^{\<i\>}, \xb  \> \ = \  \big(\hb_{\iP}\big)_i \ \ \text{ for all } \xb \in \iF^{\<i\>},
	\end{align}
	where the first equality hold because we have \. $\iF^{\<i\>} \cap \iF^{\<j\>} \neq \varnothing$ \.
	by the assumption that \. $(i,j) \in \iJ$.
	The lemma now follows by combining equations~\eqref{eq:Fij} and~\eqref{eq:Fij-sup}.
\end{proof}

\smallskip

The next lemma relates the volume of a polytope to the volumes of its facets.

\smallskip

\begin{lemma}[{\rm cf.~\cite[Lemma~5.1.1]{Sch}}{}]\label{lem:volume decomposition}
	Let \. $\iP \in \Ac$ \. be an $m$-dimensional  polytope.
	Then
	\begin{equation}\label{eq:Sch2}
		\Vol_m(\iP) \ = \ \frac{1}{m} \. \sum_{i=1}^{\ar} \. \big(\hb_{\iP}\big)_i  \. \Vol_{m-1}\big(\iF^{\<i\>}\big).
	\end{equation}
\end{lemma}

\smallskip

\begin{proof}
	The case $m=1$ is trivial,
	 so we assume that $m \geq 2$.
	We first show that
	\begin{equation}\label{eq:Sch1}
		\sum_{i=1}^{\ar} \. \Vol_{m-1}\big(\iF^{\<i\>}\big) \. \ub^{\<i\>} \ = \ \0.
	\end{equation}
	Let \ts $\zb  \in \iW \cap \Sb^{n-1}$ \ts be an arbitrary unit vector in \ts $\iW$, and
	let \ts $\iP'$ \ts be the  orthogonal projection of \ts $\iP$ \ts onto \ts $\zb^\perp$.
	Then:
	\[ \Vol_{m-1}(\iP') \ = \  \sum_{\<\zb, \ub^{\<i\>}\> \geq 0} \.
	\<\zb, \ub^{\<i\>}\> \, \Vol_{m-1}\big(\iF^{\<i\>}\big) \ = \   - \sum_{\<\zb, \ub^{\<i\>}\> < 0} \.
	\<\zb, \ub^{\<i\>}\> \, \Vol_{m-1}\big(\iF^{\<i\>}\big). \]
	This implies that
	\[  \sum_{i=1}^{\ar} \.
	\<\zb, \ub^{\<i\>}\> \, \Vol_{m-1}\big(\iF^{\<i\>}\big) \ = \ 0. \]
	Since \ts $\zb$ \ts is arbitrary, the equation~\eqref{eq:Sch1} follows.
	
	From \eqref{eq:Sch1}, we see that the right side of \eqref{eq:Sch2} does not change under translations of $\iP$.
	Since this is also true for the left side of \eqref{eq:Sch2}, we may assume that  the origin~$\0$ is contained in the interior of $\iP$.
	Then \ts $\iP$ \ts is the union of the pyramids \. $\text{conv}\big(\iF^{\<i\>} \cup \{\0\}\big)$, \. $1\le i \le \ar$, which have disjoint interiors.  This implies the equation~\eqref{eq:Sch2}.
\end{proof}

\smallskip

Let $\iP_1,\ldots, \iP_{m} \in \Ac$.
By a slight abuse of notation,
we write \. $\iV\big(\iF_{1}^{\<i\>}, \ldots, \iF_{m-1}^{\<i\>}\big)$ \. to denote the \ts
$(m-1)$-dimensional mixed volume of the facets translated into the \ts
$(m-1)$-dimensional subspace \. $\iW^{\<i\>}$.
Similarly, for every \ts $(i,j) \in \iJ$, we write
\. $\iV\big(\iF_{1}^{\<ij\>}, \ldots, \iF_{m-2}^{\<ij\>}\big)$ \. to denote the \ts $(m-2)$-dimensional
mixed volume of the faces translated into the \ts $(m-2)$-dimensional subspace \. $\iW^{\<i\>} \cap \iW^{\<j\>}$.

\smallskip

The next lemma relates the mixed volumes of polytopes to the mixed volumes of their facets.


\smallskip

\begin{lemma}[{\rm cf.~\cite[Lemma~5.1.1]{Sch}}{}]\label{lem:mixed volume decomposition}
	Let $m\geq 1$, let \ts $\Ac$ \ts be an $a$-type  of $\iW$, and let
	\ts $\iP_1,\ldots, \iP_{m} \in \Ac$.
	Then
	\begin{equation}\label{eq:Sch3}
		\iV\big(\iP_1,\ldots, \iP_{m}\big) \ = \ \frac{1}{m} \. \sum_{i=1}^{\ar} \. \big(\hb_{\iP_1}\big)_i  \. \iV\big(\iF^{\<i\>}_2,\ldots, \iF^{\<i\>}_{m}\big).
	\end{equation}
\end{lemma}

\smallskip

\begin{proof}
	We use induction over \ts $m\ge 1$.
	The case $m=1$ is trivial, so we assume that \ts $m\geq 2$, and that the lemma holds for~$(m-1)$.
	Denote the RHS of \eqref{eq:Sch3} by \. $\iU(\iP_1, \ldots, \iP_{m})$.
	By Lemma~\ref{lem:volume decomposition}, for all \. $\lambda_1,\ldots,\lambda_{m}> 0$, we have:
	\begin{align*}
		& \Vol_{m}(\lambda_1 \iP_1 + \ldots + \lambda_m \iP_m) \ = \  \frac{1}{m} \. \sum_{i=1}^{\ar} \. \big(\lambda_1 \hb_{\iP_1} + \ldots + \lambda_m \hb_{\iP_m}\big)_i \, \Vol_{m-1}\big(\lambda_1 \iF_1^{\<i\>} + \ldots + \lambda_m \iF_m^{\<i\>}\big).
	\end{align*}
	Applying \eqref{eq:mixed volume definition} to every term in the RHS of this equation, we obtain:
	\begin{align*}
		\Vol_{m}(\lambda_1 \iP_1 + \ldots + \lambda_m \iP_m)  \ & = \ \frac{1}{m} \. \sum_{i=1}^{\ar} \. \sum_{r=1}^{m} \lambda_{r}  \big( \hb_{\iP_{r}}\big)_i \. \sum_{j_2,\ldots, j_m=1}^{m}  \iV\big(\iF^{\<i\>}_{j_2},\ldots, \iF^{\<i\>}_{j_m}\big) \. \lambda_{j_2} \cdots \lambda_{j_m} \\
		& = \   \sum_{j_1, \ldots, j_m=1}^{m} \iU\big(\iP_{j_1},\ldots, \iP_{j_m}\big) \. \lambda_{j_1} \cdots \lambda_{j_m}.
	\end{align*}
	Hence it suffices to show that \ts $\iU(\iP_1,\ldots, \iP_m)$ \ts is symmetric in its arguments.

	Let \. $\hb_{\iP_1} =  (\ah_1,\ldots, \ah_\ar)$ \. be the support vector of \ts $\iP_1$.
	By the induction hypothesis, we have:
	\begin{align*}
		\iU(\iP_1,\iP_2, \iP_3,\ldots, \iP_m) \ & = \ \frac{1}{m} \. \sum_{i=1}^{\ar} \. \ah_i  \. \iV(\iF^{\<i\>}_2,\ldots, \iF^{\<i\>}_{m})
		\\ & = \ \frac{1}{m(m-1)} \.  \sum_{i=1}^{\ar} \, \ah_i \sum_{(i,j) \in \iJ} \. \Big(\hb_{\iF_2^{\<i\>}}\Big)_{j} \.  \iV(\iF^{\<ij\>}_3, \ldots, \iF^{\<ij\>}_{m}).
	\end{align*}
	This implies:
	\begin{equation}\label{eq:U1}
		\iU(\iP_1,\iP_2, \iP_3,\ldots, \iP_m) \ = \  \frac{1}{m(m-1)} \.
		\sum_{\substack{(i,j) \in \iJ\\ i<j}} \bigg[ \ah_i \Big(\hb_{\iF_2^{\<i\>}}\Big)_{j}  \. + \. \ah_j \Big(\hb_{\iF_2^{\<j\>}}\Big)_{i} \bigg] \.  \iV\big(\iF^{\<ij\>}_3, \ldots, \iF^{\<ij\>}_{m}\big).
	\end{equation}
	Now let \ts $\gb:=\hb_{\iP_2}$ \ts be the support vector of \ts $\iP_2$.
	It follows from Lemma~\ref{lem:h-vector hereditary} that
	\begin{align*}
		\ah_i \Big(\hb_{\iF_2^{\<i\>}}\Big)_{j}  \. + \. \ah_j \Big(\hb_{\iF_2^{\<j\>}}\Big)_{i} \ = \  \big(\ah_i \ag_j + \ah_j \ag_i\big) \csc \theta^{\<ij\>} \ - \ \big( \ah_i \ag_i + \ah_j \ag_j \big) \cot \theta^{\<ij\>},
	\end{align*}
	and this expression is symmetric in \. $\hb$ \ts and \ts $\gb$.
	Together with \eqref{eq:U1}, this implies that
	\begin{equation*}
		\iU\big(\iP_1,\iP_2, \iP_3,\ldots, \iP_m\big) \ = \  \iU\big(\iP_2,\iP_1, \iP_3,\ldots, \iP_m\big).
	\end{equation*}
Since  \ts $\iU(\iP_1,\ldots, \iP_m)$ \ts is symmetric in \. $\iP_2,\ldots, \iP_m$ \. by definition,
this implies the induction claim and finishes the proof of the lemma.
\end{proof}

\smallskip

We remark that Lemma~\ref{lem:mixed volume decomposition} can be extended to all convex bodies by continuity, see e.g.\ \cite[Eq.~(5.19)]{Sch} for details.

\medskip

\subsection{Mixed volume matrices}\label{subsec:mixed volume matrices}
In notation above, let \ts $m\geq 2$, and let \. $\iP_1,\ldots, \iP_{m-2} \in \Ac$ \. be simple, strongly isomorphic
$m$-dimensional polytopes. The \defnb{mixed volume matrix} \. $\bM= (\aM_{ij})_{i,j \in [\ar]}$ \. is the $\ar \times \ar$ matrix given by
\begin{alignat*}{2}
	\aM_{ii} \ &:= \ -\ts (m-2)! \. \sum_{ (i,j) \in \iJ} \. \cot \theta^{\<ij\>} \ \iV\big(\iF_{1}^{\<ij\>}, \ldots, \iF_{m-2}^{\<ij\>}\big) \qquad && \text{for \ $1\le i \le \ar$,}\\
	\aM_{ij} \ &:= \  (m-2)! \,\ts \csc \theta^{\<ij\>} \, \iV\big(\iF_{1}^{\<ij\>}, \ldots, \iF_{m-2}^{\<ij\>}\big) \qquad &&\text{for \ $i\ne j$, \ $(i,j) \in \iJ$,}\\
		\aM_{ij} \ &:= \  0 \qquad &&\text{for \ $i\ne j$, \  $(i,j) \notin \iJ$.}
\end{alignat*}
Note that \. $\bM \ts = \ts \bM(\iW, \Ac, \iP_1,\ldots, \iP_{m-2})$ \. is a symmetric matrix with nonnegative nondiagonal entries.

\smallskip

The next lemma relates  the mixed volume matrix to the mixed volume of the corresponding polytopes.
Following the notation above, for a  polytope \ts $\iA \in \Ac$, denote by  \ts $\iF^{\<i\>}_{\iA}$ \ts
the facet of $\iA$ that corresponds to the normal direction $\ub^{\<i\>}$, for all \. $1\le i \le\ar$.

\smallskip

\begin{lemma}\label{lem:mixed volume inner product}
	Let $m \geq 2$, and let $\iA, \iB \in \Ac$ be simple strongly isomorphic \ts $m$-dimensional polytopes.
	Then:
	\begin{align}
		\label{eq:mv inner product 1}	 	  \big( \bMr\hb_{\iA} \big)_{i} \ &= \ (m-1)! \,\ts \iV\big(\iF_{\iA}^{\<i\>},\iF_{1}^{\<i\>}, \ldots, \iF_{m-2}^{\<i\>}\big) \quad \text{ for all } \ 1\le i \le \ar, \. \text{ and } \\
		\label{eq:mv inner product 2}  \. \big\< \hb_{\iA},   \bMr \hb_{\iB} \big\> \ &= \ m! \, \ts \iV(\iA,\iB, \iP_1, \ldots, \iP_{m-2}).
	\end{align}
\end{lemma}

\smallskip

\begin{proof}
	For the first part, we have:
	\begin{align*}
		& \big( \bM\hb_{\iA} \big)_{i} \ = \ (m-2)! \, \sum_{(i,j) \in \iJ} \. \bigg[ \big(\hb_{\iA}\big)_j \.\csc \theta^{\<ij\>} \ - \ \big(\hb_{\iA}\big)_i \.\cot \theta^{\<ij\>} \bigg] \. \iV\big(\iF^{\<ij\>}_1, \ldots, \iF^{\<ij\>}_{m-2}\big) \\
		& =_{\text{Lem}~\ref{lem:h-vector hereditary}} \ (m-2)! \,\. \sum_{(i,j) \in \iJ} \. \Big(\hb_{\iF^{\<i\>}_{\iA}}\Big)_{j} \, \iV\big(\iF^{\<ij\>}_1, \ldots, \iF^{\<ij\>}_{m-2}\big) \\
		& =_{\text{Lem}~\ref{lem:mixed volume decomposition}} \ (m-1)! \,\. \iV\big(\iF_{\iA}^{\<i\>},\iF_{1}^{\<i\>}, \ldots, \iF_{m-2}^{\<i\>}\big).
	\end{align*}
	For the second part, we similarly have:
	\begin{align*}
		\big\< \hb_{\iA},   \bM \hb_{\iB} \big\>  \ &= \ \sum_{i=1}^{\ar}  \big(\hb_{\iA}\big)_i \. \big( \bM \hb_{\iB}\big)_i \ =_{\eqref{eq:mv inner product 1}} \ (m-1)! \,\. \sum_{i=1}^{\ar} \. \big(\hb_{\iA}\big)_i \.\iV\big(\iF_{\iB}^{\<i\>},\iF_{1}^{\<i\>}, \ldots, \iF_{m-2}^{\<i\>}\big) \\
		&=_{\text{Lem}~\ref{lem:mixed volume decomposition}} \
		m! \, \. \iV\big(\iA,\iB, \iP_1, \ldots, \iP_{m-2}\big),
	\end{align*}
	as desired.
\end{proof}

\medskip

\subsection{Combinatorial atlas for the Alexandrov--Fenchel inequality}
Let \ts $m\geq 3$.  By translating the polytopes
\. $\iP_1,\ldots, \iP_{m-2} \in \Ac$ \. if necessary, without loss of generality we can
assume that all \ts $\iP_i$ \ts contain the origin~$\0$ in the interior.
We associate to this data a combinatorial atlas
\ts $\AA:= \AA(\iW, \Ac, \iP_1,\ldots, \iP_{m-2})$ \ts of dimension \ts $\ar$, as follows.
Consider an acyclic digraph \. $\Qf=(\Vf,\Ef)$, where \.
\begin{align*}
	\Vf^+ \, := \,  \{ \vf  \.\}, \qquad \Vf^0 \, := \, \big\{ \vf^{\<1\>},\ldots, \vf^{\<\ar\>} \big\},
\quad \text{and} \quad\Ef \, := \, \big\{ (\vf, \vf^{\<1\>}), \ldots, (\vf, \vf^{\<\ar\>}) \big\}.
\end{align*}
In other words, the digraph~$\Qf$ has one source vertex \ts $\vf$ \ts connected to \ts $\ar$ \ts sink vertices
\ts $\vf^{\<1\>}, \ldots \vf^{\<\ar\>}$.
Let the associated matrix and associated vector of the source vertex \ts $\vf$ \ts be given by
\[ \bM \, = \, \bM_{\vf} \ := \ \bM(\iW, \Ac, \iP_1,\ldots, \iP_{m-2}), \qquad \hb \, = \, \hb_{\vf} \ := \ \hb_{\iP_1}. \]
Similarly, let associated matrix of sink vertices \ts $\vf^{\<i\>}$ \ts be given by
\[  \bM^{\<i\>} \, = \,  \bM_{\vf^{\<i\>}} \ := \ \bM\bigl(\iW^{\<i\>}, \Ac^{\<i\>}, \iF_2^{\<i\>},\ldots, \iF_{m-2}^{\<i\>}\bigr), \]
 where \ts $\iW^{\<i\>}\ssu \iW$ \ts is the hyperplane s.t.\ \. $\iW^{\<i\>}\bot\ub^{\<i\>}$, and where \ts $\Ac^{\<i\>}$ \ts
 is the $a$-type for $(m-1)$-dimensional polytopes \. $\iF_2^{\<i\>},\ldots, \iF_{m-2}^{\<i\>}$ \. which are strongly isomorphic by Lemma~\ref{lem:ssi hereditary}. Note that polytopes \ts $\iF_j^{\<i\>}$ \ts
 can be assumed to be contained in \ts $\iW^{\<i\>}$ \ts by translating them if necessary.
 Note also that \ts $\bM$,  \ts $\hb$ \ts and \ts $\bM^{\<i\>}$ \ts are well-defined since \ts $m\geq 3$.

 For each edge $(\vf, \vf^{\<i\>})$, the associated linear transformation
 $\bT^{\langle i \rangle}: \ts \Rb^{\ar} \to \Rb^{\ar}$
 is given by
 \[  \big(\bT^{\<i\>} \vb\big)_{j} \ := \
 \begin{cases}
 	   \av_j \, \csc \theta^{\<ij\>} \ - \   \av_i \, \cot \theta^{\<ij\>} & \ \text{ if } \ (i,j) \in \iJ,\\
 	   0 & \ \text{ otherwise}.
 \end{cases}
  \]

 \smallskip

 We now verify conditions in Theorem~\ref{t:Hyp} through the following series of lemmas.

 \smallskip

 \begin{lemma}\label{lem:mixed volume inner product child}
 	Let \. $\iA, \iB \in \Ac$ \. be simple strongly isomorphic \ts $m$-dimensional polytopes.
 	Then:
 	\begin{align}
 	 \label{eq:mv inner product 3}	 \big\< \ts\bT^{\<i\>} \hb_{\iA},  \ts \bMr^{\<i\>}  \bT^{\<i\>} \hb_{\iB} \big\> \ &= \ (m-1)! \,\.  \iV\big(\iF_{\iA}^{\<i\>},\iF_{\iB}^{\<i\>}, \iF_{2}^{\<i\>}, \ldots, \iF_{m-2}^{\<i\>}\big),
 	\end{align}
 for every \ts $1\le i \le\ar$.
 \end{lemma}

 \smallskip

 \begin{proof}
 	It follows from Lemma~\ref{lem:h-vector hereditary} and  the definition of \ts $\bT^{\<i\>}$, that
 	\[  \bT^{\<i\>} \hb_{\iA} \ = \ \hb_{\iF^{\<i\>}_{\iA}}  \qquad \text{ and } \qquad \bT^{\<i\>} \hb_{\iB} \ = \ \hb_{\iF^{\<i\>}_{\iB}}.\]
 	The conclusion of the lemma now follows by applying \eqref{eq:mv inner product 2} to $\bM^{\<i\>}$.
 \end{proof}

 \smallskip

 \begin{lemma}\label{lem:AF-Inh}
 	Combinatorial atlas \ts $\AA$ \ts defined above satisfies \eqref{eq:Inh}.
 \end{lemma}

 \smallskip

 \begin{proof}
 By linearity and by Lemma~\ref{lem:h-vector span}, it suffices to prove \eqref{eq:Inh}
 when \. $\vb= \hb_{\iA}$ \. is the support vector of some simple polytope \. $\iA \in \Ac$.
 For every \ts $i \in [\ar]$, we have:
 	\begin{align*}
 		& (\bM \vb )_{i} \ = \ (\bM \hb_{\iA} )_{i} \ =_{\eqref{eq:mv inner product 1}} \  (m-1)! \,\ts \iV\big(\iF_{\iA}^{\<i\>},\iF_{1}^{\<i\>}, \ldots, \iF_{m-2}^{\<i\>}\big).
 	\end{align*}
 On the other hand, we also have
 	\begin{align*}
 		\big \langle   \bT^{\<i\>}\vb, \, \bM^{\<i\>}  \bT^{\<i\>}\hb \big \rangle \ = \ \big\< \bT^{\<i\>} \hb_{\iA},   \bM^{\<i\>}  \bT^{\<i\>} \hb_{\iP_1} \big\>
 		 \  & =_{\eqref{eq:mv inner product 3}} \  (m-1)! \,\ts \iV\big(\iF_{\iA}^{\<i\>},\iF_{1}^{\<i\>}, \ldots, \iF_{m-2}^{\<i\>}\big),
 	\end{align*}
as desired.
 \end{proof}

\smallskip

\begin{lemma}\label{lem:AF-Irr}
	 	The associated matrix \. $\bMr=\bMr_{\vf}$ \. in the atlas \ts $\AA$ \ts is irreducible.
\end{lemma}

\smallskip

\begin{proof}
	It follows from Lemma~\ref{lem:mixed volume positivity} and the definition of $\bM$ that,
	for every distinct $i,j \in [\ar]$,
	we have \. $\aM_{ij}>0$ \. if and only if $(i,j) \in \iJ$.  The lemma now states that
the graph \ts $G=([\ar],\iJ)$ \ts is connected.  To see this, observe that \ts $G$ \ts is the facet graph
of every polytope \ts $A \in \Ac$, and thus connected.
\end{proof}

\smallskip

\begin{lemma}\label{lem:AF-h-Pos}
	Vectors \. $\hb_{\vf}$ \. and \. $\bMr_{\vf}\hb_{\vf}$ \. are strictly positive.
\end{lemma}

\smallskip

\begin{proof}
	The strict positivity of \. $\hb_v=\hb_{\iP_1}$  \. follows from the assumption that the origin is contained in the interior of~$\iP_1$.
	Now, by \eqref{eq:mv inner product 1}, we have:
	\[  \big( \bM_v\hb_v\big)_{i} \ = \ (m-1)! \,\. \iV\big(\iF_{1}^{\<i\>},\iF_{1}^{\<i\>},\iF_{2}^{\<i\>}, \ldots, \iF_{m-2}^{\<i\>}\big)  \]
	for every \ts $1\le i \le \ar$.  By Lemma~\ref{lem:mixed volume positivity}, the RHS is strictly positive.
\end{proof}

\smallskip

In particular, Lemma~\ref{lem:AF-h-Pos} implies that \ts $\supp(\bM_v)=[\ar]$.

\smallskip

\begin{lemma}\label{lem:AF-Pull}
	Combinatorial atlas \ts $\AA$ \ts satisfies \eqref{eq:PullEq}.
\end{lemma}

\smallskip

\begin{proof}
	Let us to show that
	\begin{equation}\label{eq:AF-Pull 1}
			\sum_{i =1}^{\ar} \. \ah_{i} \, \big \langle   \bT^{\<i\>}\vb, \, \bM^{\<i\>} \,  \bT^{\<i\>}\wb  \big \rangle  \ = \ \langle \vb, \bM \wb \rangle \qquad \text{ for every } \vb,\wb \in \Rb^{\ar},
	\end{equation}
Then, by substituting \. $\wb \gets \vb$ \. and the fact that \. $\supp(\bM)=[\ar]$, the equation \eqref{eq:PullEq} follows.

By bilinearity and Lemma~\ref{lem:h-vector span}, it suffices to prove \eqref{eq:AF-Pull 1} for \. $\vb=\hb_{\iA}$ \. and \. $\wb=\hb_{\iB}$, where \ts $\iA, \iB \in \Ac$ \ts are simple strongly isomorphic polytopes.
  We have:
  \begin{align*}
  	&\sum_{i =1}^{\ar} \. \ah_{i} \, \big \langle   \bT^{\<i\>}\vb, \, \bM^{\<i\>} \,  \bT^{\<i\>}\wb  \big \rangle
  	\ = \ \sum_{i =1}^{\ar} \. \big(\hb_{\iP_1}\big)_i \, \big \langle   \bT^{\<i\>}\hb_{\iA}, \, \bM^{\<i\>} \,  \bT^{\<i\>}\hb_{\iB}  \big \rangle  \\
  	& \hskip1.cm =_{\eqref{eq:mv inner product 3}} \
  	(m-1)! \,\ts\sum_{i =1}^{\ar} \. \big(\hb_{\iP_1}\big)_i \, \iV\big(\iF_{\iA}^{\<i\>},\iF_{\iB}^{\<i\>}, \iF_{2}^{\<i\>}, \ldots, \iF_{m-2}^{\<i\>}\big).
  \end{align*}
On the other hand, we also have
\begin{equation*}
\langle \vb, \bM \wb \rangle \ = \ \langle \hb_{\iA}, \bM \hb_{\iB} \rangle \ =_{\eqref{eq:mv inner product 1}} \  m! \. \iV(\iA, \iB, \iP_1,\ldots, \iP_{m-2}).
\end{equation*}
By Lemma~\ref{lem:mixed volume decomposition}, the RHS of the two equations above are equal.
Thus so are the LHS, as desired.
\end{proof}

\medskip

\subsection{Hyperbolicity of the mixed volume matrix}

It  follows from \eqref{eq:mv inner product 2}
that the Alexandrov--Fenchel inequality is a special case of the following theorem.

\begin{thm}\label{t:AF-Hyp}
	Let \ts $\iW \subseteq \Rb^n$ \ts  be a linear subspace   of dimension \ts $m\geq 2$,
	let \ts $\Ac$ \ts be an $a$-type of $\iW$, and let
	\ts $\iP_1,\ldots, \iP_{m-2} \subset \iW$ \ts be simple, strongly isomorphic polytopes in $\Ac$.
	Then
	\[  \text{the matrix } \ \bMr(\iW, \Ac, \iP_1,\ldots, \iP_{m-2}) \  \text{satisfies \eqref{eq:Hyp}.} \]
\end{thm}

\smallskip

We  build toward the proof of Theorem~\ref{t:AF-Hyp}.
Our first step is built on  the  Brunn--Minkowski inequality for $\Rb^2$ (note that this inequality does not assume that the polytopes are convex).

\smallskip

\begin{thm}[{\rm Brunn--Minkowski inequality in $\Rb^2$}{}]\label{thm:BM}
	Let \. $\iA, \ts\iB\ssu \rr^2$ \. be two convex polygons in the plane.
	Then
	\[ \sqrt{\area(\iA+\iB)} \ \geq \ \sqrt{\area(\iA)} \, + \, \sqrt{\area(\iB)}. \]
\end{thm}

\smallskip

This inequality is classical and is especially easy to prove in the plane.
For completeness, we include a short proof below, in~$\S$\ref{ss:AF-BM}.

\smallskip

\begin{lemma}\label{lem:AF-base}
	Theorem~\ref{t:AF-Hyp} holds for \ts $m=2$.
\end{lemma}

\smallskip

\begin{proof}
	First, observe that the Brunn--Minkowski inequality implies the Alexandrov--Fenchel inequality in the plane.
	Indeed, let \. $\iA,\iB \ssu \rr^2$ \. be convex polygons.
	 Then:
	 \begin{equation}\label{eq:BM-app-1}
	 	\area(\la\ts\iA\. +\. \mu\ts \iB) \ =_{\eqref{eq:mixed volume definition}} \ \la^2 \. \iV(\iA,\iA) \ + \  2 \ts\la \ts\mu \. \iV(\iA,\iB) \ + \ \mu^2 \. \iV(\iB,\iB).
	 \end{equation}
	On the other hand,
	\begin{equation}\label{eq:BM-app-2}
		\Big[\sqrt{\area(\la\ts\iA)} \, + \, \sqrt{\area(\mu\ts\iB)} \Big]^2  \ = \ \la^2 \. \area(\iA) \ + \  2\ts \la\ts \mu \. \sqrt{\area(\iA) \. \area(\iB)} \ + \ \mu^2 \. \area(\iB).
	\end{equation}
Taking the difference of \eqref{eq:BM-app-1} and \eqref{eq:BM-app-2} and applying Theorem~\ref{thm:BM}, we conclude that
\begin{equation}\label{eq:AF-base}
	\iV(\iA,\iB)^2 \ \geq \ \iV(\iA,\iA)\, \iV(\iB,\iB),
\end{equation}
as desired.

Now, let us show that the associated matrix \. $\bM:=\bM(\iW,\Ac)$ \. satisfies \eqref{eq:NDC}.  By
Lemma~\ref{l:Hyp is OPE}, this implies \eqref{eq:Hyp} and proves the result.
Let \ts $\gb = \hb_{\iP}$ \ts for some convex polygon \ts $\iP \in \Ac$ \ts with \ts $\ar$ \ts edges.
Note that \. $\< \gb, \bM \gb\> \ = \ 2 \. \area(\iP) >0$ by \eqref{eq:mv inner product 2}.
Let \ts $\vb \in \Rb^{\ar}$ \ts be an arbitrary vector satisfying \. $\<\vb, \bM \gb\>=0$.
By Lemma~\ref{lem:h-vector span}, there exists \. $c\in \Rb$ \. and a convex polygon~$\iQ$
strongly isomorphic to $\iP$,  such that  \. $\vb = c(\hb_{\iQ} - \hb_{\iP})$.
Note that polygon \ts $\iQ$ \ts has \ts $\ar$ \ts edges parallel to the corresponding
edges of~$\iP$.

Now observe that \. $\<\vb, \bM \gb\>=0$ \. is equivalent to
\. $\<\hb_{\iP}, \bM \hb_{\iP}\>  \ = \  \<\hb_{\iQ}, \bM \hb_{\iP}\>$,
which by \eqref{eq:mv inner product 2} is equivalent to
\begin{equation}\label{eq:AF-base-2}
	\iV(\iP,\iP) \ = \ \iV(\iP,\iQ).
\end{equation}
Together with \eqref{eq:AF-base}, this implies that
\begin{equation}\label{eq:AF-base-3}
	\iV(\iP,\iQ) \ \geq  \ \iV(\iQ,\iQ).
\end{equation}
Now note that
\begin{align*}
	\<\vb, \bM \vb \> \ &= \  c^2 \big( \<\hb_{\iP}, \bM \hb_{\iP}\> \ + \ \<\hb_{\iQ}, \bM \hb_{\iQ}\>  \ -  \ 2 \.\<\hb_{\iP}, \bM \hb_{\iQ}\>\big) \\
	\ &=_{\eqref{eq:mv inner product 2}} \  c^2 \big(\iV(\iP,\iP) \ + \ \iV(\iQ,\iQ)  \ -  \ 2 \iV(\iP,\iQ) \big).
\end{align*}
The RHS is nonpositive by \eqref{eq:AF-base-2} and \eqref{eq:AF-base-3}, and the proof is complete.
\end{proof}

\smallskip

\begin{proof}[Proof of Theorem~\ref{t:AF-Hyp}]
We prove the theorem by induction on $m$.
The case $m=2$ has been proved in Lemma~\ref{lem:AF-base},
so we can assume that \ts $m\geq 3$.
The atlas \. $\AA = \AA(\iW, \Ac, \iP_1,\ldots, \iP_{m-2})$ \.
satisfies \eqref{eq:Inh} and \eqref{eq:PullEq} by Lemma~\ref{lem:AF-Inh}
and Lemma~\ref{lem:AF-Pull}, respectively.  Note that~$\vf$ is a
regular vertex by Lemma~\ref{lem:AF-Irr} and Lemma~\ref{lem:AF-h-Pos},
and that every out-neighbor \ts $\vf^{\<i\>}$ \ts of \ts $\vf$ \ts
satisfies \eqref{eq:Hyp} by the induction assumption.
The conclusion of the theorem now follows from Theorem~\ref{t:Hyp}.
\end{proof}

\smallskip

 \begin{lemma}[{\rm cf.~\cite[Lemma~2.4.12]{Sch}}{}]\label{l:SIP}
Let \. $\iP, \iQ  \in \Rb^{n}$ \. be convex polytopes, and suppose we have \ts
$\la,\la',\mu,\mu'>0$.  Then polytopes \. $\la\ts \iP \. + \. \mu\ts \iQ$ \.  and
\. $\la'\ts \iP \. + \. \mu'\ts \iQ$ \. are strongly isomorphic.
 \end{lemma}

\smallskip

\begin{proof}
By the definition of the Minkowski sum, observe that
\[ \dim \iF(\iA\ts +\ts\iB,\ub) \ = \ \dim \ts\bigl(\ts \iF(\iA,\ub)+\iF(\iB,\ub)\ts\bigr)  \ = \ \dim \rr\big\<\iF(\iA,\ub),\iF(\iB,\ub)\big\>,
\]
for all convex polytopes \. $\iA, \iB\in \rr^n$ \.
and for all unit vectors \. $\ub \in \Sb^{n-1}$.
Note that the last equality holds only if both \ts $\iF(\iA,\ub)$ \ts and \ts $\iF(\iB,\ub)$ \ts contain the origin,
which we can assume without loss of generality by translating the polytopes.
Since the dimension of the subspace
is invariant under scaling of the vectors, the result follows.
\end{proof}

\smallskip

\begin{proof}[Proof of Theorem~\ref{thm:AF}]
Let \. $\iA, \iB,\iP_1,\ldots, \ldots \iP_{n-2}  \in \Rb^{n}$ \.
be arbitrary simple strongly isomorphic polytopes with $a$-type $\Ac$.
By Theorem~\ref{t:AF-Hyp} with \ts $m= n$ \ts and \ts $\iW\ = \Rb^{n}$,
combined with~\eqref{eq:mv inner product 2}, we obtain
the Alexandrov--Fenchel inequality~\eqref{eq:AF} in this case:
\begin{equation}\label{eq:AF-polytope}
	 \iV(\iA,\iB, \iP_1, \ldots, \iP_{n-2})^2  \ \geq \ \iV(\iA,\iA, \iP_1, \ldots, \iP_{n-2}) \, \iV(\iB,\iB, \iP_1, \ldots, \iP_{n-2}).
\end{equation}

Suppose now that \. $\iA, \iB,\iP_1,\ldots, \ldots \iP_{n-2}  \in \Rb^{n}$ \.
are general simple convex polytopes.  Let \ts $\ve>0$, and define
$$\iQ \ := \ \iA \. + \. \iB \. + \. \iP_1 \. + \. \ldots \. +  \. \iP_{n-2} \..
$$
By Lemma~\ref{l:SIP}, polytopes \. $\iA+\ve\ts\iQ$, \. $\iB+\ve\ts\iQ$, \.
$\iP_{1}+\ve\ts\iQ$, \ldots, \ts $\iP_{n-2} +\ve\ts\iQ$ \. are all strongly isomorphic.
Note that they are not necessarily simple; in that case use \ts $\iQ \gets \iQ+\iQ'$ \ts
where \ts $\iQ'$ \ts is a generic polytope obtained as a Minkowski sum of vectors orthogonal
to unit vectors \ts $\ub$ \ts for which \ts $\iF(\iQ,\ub)$ \ts is a non-simple vertex.
Taking the limit \ts $\ve\to 0$ \ts and using monotonicity of mixed volumes,
we obtain~\eqref{eq:AF-polytope} for general polytopes.  We omit the details.

Finally, recall that general convex bodies can be approximated to an arbitrary precision
by collections of convex polytopes. The theorem now follows by taking continuous limits
of \eqref{eq:AF-polytope}.
\end{proof}

\medskip

\subsection{Proof of the Brunn--Minkowski inequality in the plane}\label{ss:AF-BM}
For completeness, we include a simple proof of Theorem~\ref{thm:BM} by induction
which goes through non-convex regions and uses a limit argument at the end.

\smallskip

A \defn{brick} is an axis-parallel rectangle \. $[x_1,x_2]\times [y_1,y_2]$,
for some \. $x_1<x_2$ \. and \. $y_1< y_2$.  A \defn{brick region} is a
union of finitely many brick, disjoint except at the boundary.  Note that
brick regions are not necessarily convex.

\smallskip

\begin{lemma}  \label{l:BM-box}
Brunn--Minkowski inequality holds for bricks in the plane.
\end{lemma}

\begin{proof}  Let \. $\iA, \iB\ssu\rr^2$ \. be bricks with side
lengths \. $(a_1,a_2)$ \. and \. $(b_1,b_2)$, respectively.
The Brunn--Minkowski inequality in this case states:
\begin{equation}\label{eq:BM-box}
\sqrt{(a_1\ts +\ts b_1)(a_2\ts +\ts b_2)} \ \ge \ \sqrt{a_1\ts b_1} \, + \, \sqrt{a_2\ts b_2}\,.
\end{equation}
Squaring both sides gives
$$
(a_1\ts +\ts b_1)(a_2\ts +\ts b_2) \ \ge \ a_1\ts b_1 \, + \, a_2\ts b_2 \, + \, 2\. \sqrt{a_1\ts b_1\ts a_2\ts b_2} \,,
$$
which in turn follows from the AM-GM inequality:
$$
a_1\ts b_2 \, + \, b_1\ts a_2 \ \ge \ 2\. \sqrt{a_1\ts b_1\ts a_2\ts b_2} \,.
$$
\end{proof}

\smallskip

\begin{lemma}  \label{l:BM-boxed}
Brunn--Minkowski inequality holds for brick regions in the plane.
\end{lemma}

\begin{proof}  Let \. $\iA, \iB\ssu\rr^2$ \. be brick regions.
We use induction on the total number \ts $k\ge 2$ \ts of the bricks in
both regions.  When $k=2$ the result is given by Lemma~\ref{l:BM-box},
so we can assume $k\ge 3$.  Then one of the regions, say~$\iA$,
has at least two bricks.

Denote by \. $H$ \. the axis-parallel line which separates some bricks
in~$\iA$.  Denote by \ts $\iA_1$ \ts and \ts $\iA_2$ \ts brick regions
of~$\iA$ separated by~$H$, and let \ts $\theta:=\area(\iA_1)/\area(\iA)$.
We can always move \ts $\iA$ \ts so that \ts $H$ \ts contains the origin,
and then move
\ts $\iB$ \ts so that \ts $H$ \ts separates \ts
$\iB$ \ts into two brick regions \ts $\iB_1$ \ts and \ts $\iB_2$ \ts with
the same ratio: \ts $\area(\iB_1)/\area(\iB)=\theta$.  See an example in
Figure~\ref{f:BM-brick}.

\begin{figure}[hbt]
\begin{center}
\includegraphics[width=10.8cm]{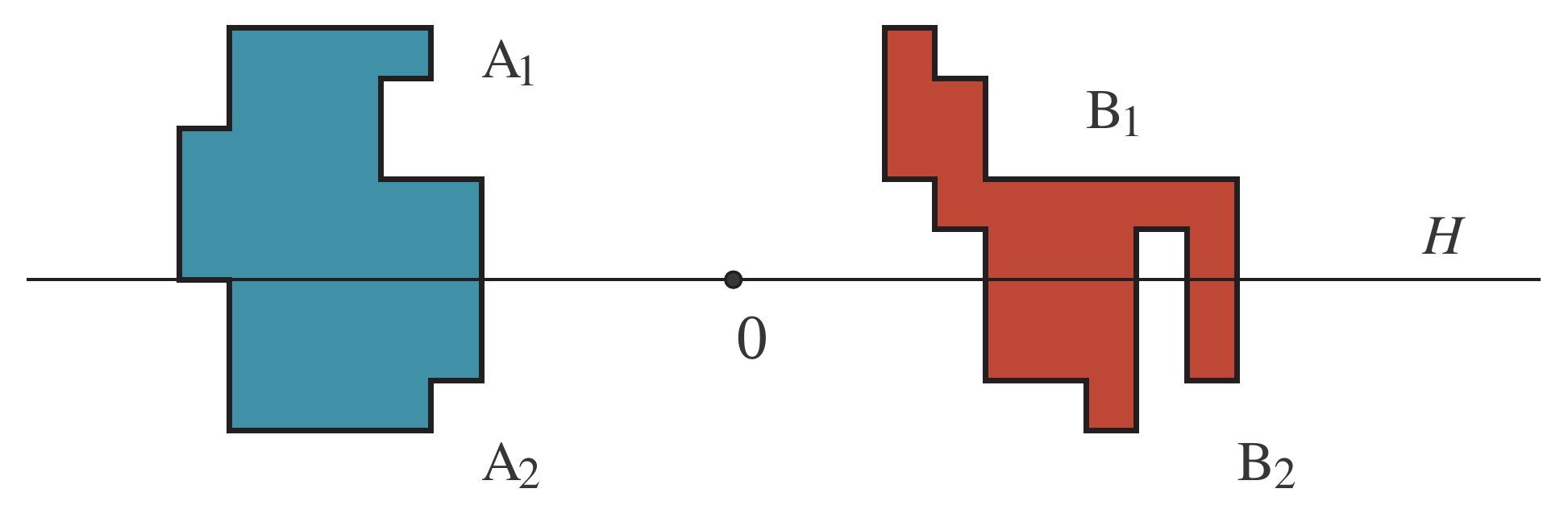}
\end{center}
\caption{Two brick regions \ts $\iA$ \ts and \ts $\iB$ \ts
divided by a line~$H$ with the same area ratios.}
\label{f:BM-brick}
\end{figure}

Observe that the combined number of bricks in \ts $(\iA_1,\iB_1)$ \ts
is smaller than~$k$, so inductive assumption applies.  The same
holds for \ts $(\iA_2,\iB_2)$.  We then have:
$$\aligned
\. & \area(\iA+\iB) \ \ge \ \area(\iA_1+\iB_1) \. + \.  \area(\iA_2+\iB_2) \\
& \qquad \ge \ \left[\sqrt{\area(\iA_1)} \. + \. \sqrt{\area(\iB_1)}\right]^2 \, + \, \left[\sqrt{\area(\iA_2)} \. + \. \sqrt{\area(\iB_2)}\right]^2\\
& \qquad\ge \ \left[\sqrt{\theta\ts\area(\iA)} \. + \. \sqrt{\theta\ts\area(\iB)}\right]^2 \, + \, \left[\sqrt{(1-\theta)\ts\area(\iA)} \. + \. \sqrt{(1-\theta)\ts\area(\iB)}\right]^2\\
& \qquad\ge \ \bigl[\theta \ts + \ts (1-\theta)\bigr] \.
\left[\sqrt{\area(\iA)} \. + \. \sqrt{\area(\iB)}\right]^2  \ = \ \sqrt{\area(\iA)} \. + \. \sqrt{\area(\iB)}\,.
\endaligned
$$
Here the first inequality follows since sets \ts $\iA_1+\iB_1$ \ts and \ts $\iA_2+\iB_2$ \ts lie on different
sides of~$H$.  The second inequality follows from induction assumption.  The remaining inequalities
are trivial equalities.  This completes the proof.
\end{proof}

\smallskip

\begin{proof}[Proof of Theorem~\ref{thm:BM}]
Let \ts $\iA, \iB\ssu \rr^2$ be two convex polygons and let \ts $\ve>0$.
Consider the scaled square grid \ts $\ve\ts \zz^2$ \ts and denote by
\ts $\iA_\ve,\iB_\ve$ \ts the unions of \ts $\ve\times \ve$ \ts squares completely
inside \ts $\iA, \iB$.  Observe that \. $\area(\iA_\ve)\to \area(\iA)$, \ts
$\area(\iB_\ve)\to \area(\iB)$, \ts and \. $\area(\iA_\ve+\iB_\ve)\to \area(\iA+\iB)$, \ts as
\ts $\ve\to 0$.  The result now follows by applying Lemma~\ref{l:BM-boxed} to
brick regions \. $(\iA_\ve,\iB_\ve)$ \. and taking the limit.
\end{proof}

\bigskip

\section{Final remarks}\label{s:finrem}

\subsection{Our sources}\label{ss:finrem-hist}
As we mentioned in the introduction, this paper is written with
expository purposes.  We present no new results except for
the tangential Theorem~\ref{t:simp} which can only be understood
in the context of the proof of Theorem~\ref{t:matroids-BH}
in Section~\ref{s:matroid}.  While the majority of the presentation
is new, some of it borrows more or less directly from other
sources.  Here is a quick reference guide.

Section~\ref{s:atlas} is almost directly lifted from~\cite{CP}.
Parts of it are strongly influenced by~\cite{BH} and~\cite{SvH19},
notably the proof of Lemma~\ref{l:Hyp is OPE}.  Section~\ref{s:matroid}
is adapted and substantially simplified from~\cite{CP}, so much that
it appears unrecognizable.  Note that we omit the equality conditions
which can be similarly adapted.

Section~\ref{s:Lorentzian} was originally intended to be included
in~\cite{CP}, but was left out when that paper exploded in size.
The aim of that section is to emphasize that the Lorentzian polynomial
approach in \cite{ALOV,BH} is a special case of ours.  There are many
indirect connections to all three papers, but the presentation here
seems novel.  Note that there are several equivalent definitions
of Lorentzian polynomials (see~\cite[\S2]{BH}), and we choose
the one closest to our context for convenience.

Section~\ref{s:AF ineq} came largely as a byproduct of our original
effort in~\cite{CP} to understand Stanley's inequality via the proof of the
Alexandrov--Fenchel inequality in~\cite{SvH19} and Stanley's
original paper~\cite{Sta}.  In an effort to make the presentation
self-contained, we borrow liberally from~\cite{Sch}.  Our presentation
of the Brunn--Minkowski inequality is standard and follows \cite[$\S$12.2]{Mat} and
\cite[$\S$7.7,\ts \S41.4]{Pak}.

Note that in our presentation, the totality of the proof of the
Alexandrov--Fenchel inequality is a rather lengthy union of
Section~\ref{s:atlas} and Section~\ref{s:AF ineq}.  There are several other
relatively recent proofs \cite{CEKMS,KK,SvH19,Wang} based on different
ideas and which employ existing technologies to a different degree.   
Obviously, the notions of ``simple'' and ``self-contained'' we used 
in the introductions are subjective, so we can only state our own view. 
Similarly, we challenge the assessment in~\cite{KK} which call their 
proof ``elementary''.  

Let us single out \cite{CEKMS} which relates the proof of the
Alexandrov--Fenchel inequality in~\cite{SvH19} and (implicitly)
the polynomial method in~\cite{BH}.  Although our work is independent
of~\cite{CEKMS}, it would be curious to see if that method can be extended
to the full power of the combinatorial atlas technology in~\cite{CP}.

Finally, let us mention that in the special case of brick polytopes,
our proof of the Alexandrov--Fenchel inequality simplifies so much
that it becomes known and elegantly presented in~\cite{vL}.  See
also~\cite{Gur} for a generalization and a modern treatment from
the Lorentzian polynomial point of view.

\smallskip

\subsection{Stanley's inequality}\label{ss:finrem-Sta}
Recall the straightforward derivation of Stanley's inequality
(for the number of certain linear extensions in posets) from
the Alexandrov--Fenchel inequality given in~\cite{Sta}.
Given the linear algebraic proof of the latter in Section~\ref{s:AF ineq},
one can ask why do we have such a lengthy proof of them in~\cite[$\S$14]{CP}.
There are two reasons for this, one technical and one structural.

The technical reason is that our approach allows us to obtain $q$-analogues
and more general deformations of Stanley's inequality, which do not seem to
follow directly from the Alexandrov--Fenchel inequality.  The structural
reason is that we really aim to rederive the equality conditions for
Stanley's inequality which were recently obtained by a difficult argument
in a breakthrough paper~\cite{SvH}.

In a nutshell, we employ self-similarity inherent to the problem,
in terms of faces of order polytopes used to translate the problem into geometry.
Applying iteratively the argument in our proof of the Alexandrov--Fenchel
inequality, allowed us to streamline the construction and make it completely
explicit if rather lengthy.  This, in turn, gave both the equality conditions
and the deformations mentions above.

It would be interesting to see if the argument along these lines can be
replicated in other cases.  In particular, the Kahn--Saks inequality in~\cite{KS}
is the closest to Stanley's inequality, and yet does not have a combinatorial
atlas proof.  Note that the equality conditions are also harder to obtain
in this case, and they are not even conjectured at this point,
see~\cite[$\S$8.3]{CPP2}.

Finally, let us emphasize that our proof of the Alexandrov--Fenchel inequality
does not extend to give the equality conditions in full generality.  There are
two reasons for this:  parallel translations of the facets and the need to take
limits.  While the former is ``combinatorial'' and is an obstacle to making the
equality characterization explicit, the latter is more critical as taking limits can
\emph{create} equalities.

To appreciate the distinction, the reader can think of convex polygons in the plane,
where the usual \emph{isoperimetric inequality} is always strict vs.\ the circle which
is the limit of such polygons, and where the isoperimetric inequality is an equality.
While the equality conditions are classically understood for the Brunn--Minkowski
inequality \cite[$\S$8]{BZ-book} and for the polytopal case of the Alexandrov--Fenchel
inequality~\cite{SvH}, for general convex bodies new ideas are needed.

\smallskip

\subsection{Simplicial complexes}
Theorem~\ref{t:abs} shows that an abstract simplicial complex is necessarily a matroid
if the corresponding combinatorial atlas satisfies~\eqref{eq:Hyp}.
In a way, this result is comparable to the equivalence between Lorentzian polynomials
and $M$-convexity in \cite[Thm~3.10~(1)--(7)]{BH}.

In fact, there are many simplicial complexes for which the sequence \. $(\aI_k)_{k \geq 0}$ \.
of number of faces of cardinality $k$, is not unimodal, let alone log-concave.
The face lattice of simplicial complexes is especially interesting and well studied.
In this case, unimodality was known as the \emph{Motzkin conjecture} (1961), which
was disproved in~\cite{Bjo}.  There, Bj\"{o}rner gave an example of a $24$-dimensional
simplicial polytopes for which the \ts \emph{$f$-vector} is not unimodal.
See also smaller examples in \cite{BL,Bjo2} of $20$-dimensional simplicial polytopes,
and~\cite{Eck} which proved that this dimension cannot be lowered.

Finally, let us mention that in Section~\ref{s:matroid} we start with a simplicial complex
and then we build an atlas, while in Section~\ref{s:Lorentzian} we build an atlas starting
with a Lorentzian polynomial.  On the other hand, in~\cite{ALOV-RW}, the authors start with
a Lorentzian polynomial and then build a simplicial complex.  The connection between
these approaches is yet to be fully understood.

\vskip.7cm

\subsection*{Acknowledgements}
We are grateful to Matt Baker, June Huh, Jeff Kahn, Gaiane Panina, Greta Panova,
Yair Shenfeld, Hunter Spink, Ramon van Handel and Cynthia Vinzant for helpful
discussions and remarks on the subject.
The first author was partially supported by the Simons Foundation.
The second author was partially supported by the~NSF.


%


\end{document}